\documentclass[a4paper,11pt]{article}
\usepackage{inputenc}
\usepackage{authblk}
\usepackage[table]{xcolor}
\usepackage{color}
\usepackage{amsmath}
\usepackage{amssymb}
\usepackage{amsthm}
\usepackage{booktabs}
\usepackage{xcolor,graphicx,float}
\usepackage{hyperref}
\usepackage{soul}
\usepackage{indentfirst}
\hypersetup{colorlinks=true,
	linkcolor=blue,
	anchorcolor=blue,
	urlcolor=black,
	citecolor=blue}
\usepackage{enumitem}
\setlist[itemize]{itemsep=0pt, topsep=0pt}
\usepackage{tikz}
\usepackage{subcaption}
\usepackage{caption}

\definecolor{yalv}{RGB}{150,194,78}
\definecolor{tonglv}{RGB}{43,174,133}
\usepackage[
a4paper,
textwidth=16cm,
textheight=23cm,
centering
]{geometry}

\theoremstyle{plain}
\newtheorem{theorem}{\bf Theorem}[section]
\newtheorem{lemma}[theorem]{\bf Lemma}

\newtheorem{corollary}[theorem]{\bf Corollary}
\newtheorem{conjecture}{\bf Conjecture}
\newtheorem{algorithm}{\bf Algorithm}

\theoremstyle{remark}

\newtheorem{construction}{\bf Construction}
\newtheorem{claim}{\bf Claim}

\newtheorem{assumption}{\bf Assumption}

\numberwithin{equation}{section}

\title{\bf Sharp stability for cross $t$-intersecting families of permutations in the linear range}
\author[1]{Jie Wen\thanks{E-mail: \text{jwen@mail.bnu.edu.cn}}}
\author[1]{Benjian Lv\thanks{Corresponding author. E-mail: \text{bjlv@bnu.edu.cn}}}
\affil[1]{\small Laboratory of Mathematics and Complex Systems (Ministry of Education), School of Mathematical Sciences, Beijing Normal University, Beijing 100875, China}
\date{}

\begin{document}
	\renewcommand{\baselinestretch}{1.2}
	\maketitle
	\begin{abstract}
		\medskip
 Two families $\mathcal{F},\mathcal{G}\subseteq S_n$ are called cross $t$-intersecting if every $\sigma\in\mathcal{F}$ and
$\tau\in\mathcal{G}$ agree on at least $t$ points. A $t$-coset is a coset of the stabilizer of $t$ points. A subset of $S_n$ is non-trivial if it is not contained in any $t$-coset. Let $d_m$ denote the $m$-th derangement number. We prove that, for all $t\geq1$ and all $n\geq400t$, every pair of cross $t$-intersecting families $\mathcal{F},\mathcal{G}\subseteq S_n$ satisfies the following:
\begin{itemize}
	\item$|\mathcal{F}||\mathcal{G}|\leq((n-t)!-d_{n-t}-d_{n-t-1})((n-t)!+t)$ if $\mathcal{F}\cup\mathcal{G}$ is non-trivial.
	\item$|\mathcal{F}||\mathcal{G}|\leq((n-t)!-d_{n-t}-d_{n-t-1}+t)^2$ if both $\mathcal{F}$ and $\mathcal{G}$ are non-trivial.
	\item$\min\{|\mathcal{F}\setminus\mathcal{C}|,|\mathcal{G}\setminus\mathcal{C}|\}\leq t((n-t-1)!-(n-t-2)!)$ for some $t$-coset $\mathcal{C}$ if $t\geq2$.
\end{itemize}
We also characterize all extremal configurations. The first result extends a theorem of Ellis (2011) to an exponentially wider range and sharpens the stability theorem of Keller, Lifshitz, Minzer and Sheinfeld (2024); the second establishes a product version of the classical Hilton--Milner--Frankl theorem for permutations; and the third settles the remaining cases $t\geq2$ of a conjecture of Ellis (2011) in a stronger form. In all three results, the linear dependence on $t$ is essentially optimal. Our proofs are based on the spread approximation method introduced by Kupavskii and Zakharov and on an approach to cross $t$-intersection problems developed by the present authors, with several essential refinements. The techniques may apply to cross-intersection problems for a variety of combinatorial objects. As an application, we prove a product version of the Hilton--Milner--Frankl theorem for the alternating group. \vspace{1em}

		
		
	\end{abstract}
	\section{Introduction}
	\subsection{Background}
Let $[n]=\{1,2,\ldots,n\}$, and write $\binom{[n]}{k}$ for the family of all $k$-subsets of $[n]$. A family $\mathcal{F}\subseteq\binom{[n]}{k}$ is \emph{$t$-intersecting} if any two of its members have at least $t$ elements in common. The celebrated Erd\H{o}s--Ko--Rado theorem \cite{Erdos-Ko-Rado-1961} states that, for fixed $k$ and $t$ and sufficiently large $n$, every such family has at most $\binom{n-t}{k-t}$ members, with equality precisely for the families consisting of all $k$-subsets containing a fixed $t$-subset. It was a major open problem to determine the least possible value $n_0(k,t)$ for $n$ such that the upper bound holds. The original paper \cite{Erdos-Ko-Rado-1961} proved $n_0(k,1)=2k$, and established $n_0(k,t)\leq(k-t)\binom{k}{t}^3+t$. Frankl \cite{Frankl-1976} made a breakthrough by determining $n_0(k,t)=(t+1)(k-t+1)$ for $t\geq 15$. Wilson \cite{Wilson-1984} subsequently showed $n_0(k,t)=(t+1)(k-t+1)$ for all $t\geq1$, via an ingenious algebraic proof. Ahlswede and Khachatrian \cite{Ahlswede-Khachatrian-1997} completely settled the problem by 
establishing the famous \emph{complete intersection theorem}, which determines $t$-intersecting families of $\binom{[n]}{k}$ with maximum size for all values of $n, k$, and $t$. Besides finding largest $t$-intersecting families, another longstanding and active direction of research concerns the problem of characterizing the structure of large $t$-intersecting families. A $t$-intersecting family $\mathcal{F}\subseteq\binom{[n]}{k}$ is \emph{non-trivial} if $|\cap_{F\in\mathcal{F}}F|<t$, that is, no $t$-subset is contained in all members of $\mathcal{F}$. The seminal Hilton--Milner--Frankl theorem \cite{Hilton-Milner-1967,Frankl-1978} determines largest non-trivial $t$-intersecting families in $\binom{[n]}{k}$.

Intersection problems arise naturally for many combinatorial objects (see, e.g., \cite{Ellis-2011,Ellis-Filmus-Friedgut-2012,Erdos-Szekely-2000,Frankl-Furedi,Frankl-Wilson-1986,Ihringer,Ku-Renshaw-2008,Liao-2024,Meagher-Sin,Moon,Scott}), and their study has inspired a variety of insightful results. For a systematic introduction to intersection problems, we refer the reader to the surveys \cite{Ellis-book,Frankl-Tokushige-2016} and the monographs \cite{Frankl-Tokushige-book,Godsil-Meagher-book}. 

In this paper, we focus on permutations, one of the most thoroughly studied objects in this area. Let us introduce some notation. A family $\mathcal{F}\subseteq S_n$ is \emph{$t$-intersecting} if any two of its members agree on at least $t$ points. A \emph{$t$-coset} is a coset of the stabilizer of $t$ points. A subset of $S_n$ is \emph{non-trivial} if it is not contained in any $t$-coset. Two families $\mathcal{F},\mathcal{G}\subseteq S_n$ are called \emph{cross $t$-intersecting} if every $\sigma\in\mathcal{F}$ and $\tau\in\mathcal{G}$ agree on at least $t$ points. We say that a pair $(\mathcal{F},\mathcal{G})$ of cross $t$-intersecting families is \emph{non-trivial} if they are not both contained in the same $t$-coset, that is, $\mathcal{F}\cup\mathcal{G}$ is non-trivial. 

In 1977, Deza and Frankl \cite{Frankl-Deza-1977} proved that every $1$-intersecting family in $S_n$ has size at most $(n-1)!$. The bound is tight, as witnessed by a $1$-coset. They further conjectured that, when $n$ is sufficiently large depending on $t$, every $t$-intersecting family has size at most $(n-t)!$. The problem turned out to be surprisingly difficult. Even for $t=1$, the characterization of extremal families remained open until 2003, when Cameron and Ku \cite{Cameron-Ku-2003} showed that the only extremal families are $1$-cosets. An independent proof was given by Larose and Malvenuto \cite{Larose-Malvenuto-2004}. A major breakthrough came in 2011, when Ellis, Friedgut and Pilpel \cite{Ellis-Friedgut-Pilpel-2011} proved that, for $n$ sufficiently large depending on $t$, every pair $(\mathcal{F},\mathcal{G})$ of cross $t$-intersecting families in $S_n$ satisfies $|\mathcal{F}||\mathcal{G}|\leq((n-t)!)^2$, with equality only if the pair is trivial. By taking $\mathcal{F}=\mathcal{G}$, they proved that the maximum-sized $t$-intersecting families in $S_n$ are precisely the $t$-cosets. Their proof uses eigenvalue techniques, together with the representation theory of $S_n$.

In 2024, Kupavskii and Zakharov \cite{Kupavskii-Zakharov-2024} introduced the powerful \emph{spread approximation method}. The method builds on the breakthrough of Alweiss, Lovett, Wu and Zhang \cite{sunflower} on the Erd\H{o}s--Rado sunflower conjecture, and it is purely combinatorial, involving the important pseudorandom notion of \emph{spreadness}. Using this method, Kupavskii and Zakharov proved the Erd\H{o}s--Ko--Rado theorem of permutations for $n/(\log n)^2=\Omega(t)$, extending the theorem of Ellis, Friedgut and Pilpel \cite{Ellis-Friedgut-Pilpel-2011} to an exponentially wider range. They also obtained significant improvements for the more difficult forbidden intersection problem for permutations, extending earlier work of Ellis and Lifshitz \cite{Ellis-Lifshitz-2022} based on the junta method and representation theory. The method has since been further developed and has led to substantial progress on a variety of open problems (see, e.g., \cite{Kupavskii-2024,Saengrungkongka,Frankl-Kupavskii-2025,Kupavskii-2026,KKLS}).

For general $t\geq2$, the linear range for intersection problems for permutations was reached by Keller, Lifshitz, Minzer and Sheinfeld \cite{KLMS}, who established the following product and stability results for cross $t$-intersecting families.
\begin{theorem}[\cite{KLMS}]\label{thmKLMS}
	There exists $c_0>0$ such that the following holds for all $t\geq1$ and all $n\geq c_0t$. Let $\mathcal{F},\mathcal{G}\subseteq S_n$ be cross t-intersecting families. Then:
	\begin{itemize}
		\item[\rm(i)] $|\mathcal{F}||\mathcal{G}|\leq(n-t)!^2$.
		\item[\rm(ii)] If $|\mathcal{F}||\mathcal{G}|\geq0.75(n-t)!^2$, then $(\mathcal{F},\mathcal{G})$ is trivial.
		\item[\rm(iii)] If $\mathcal{F}=\mathcal{G}$ and $|\mathcal{F}|\geq0.75(n-t)!$, then $\mathcal{F}$ is trivial.
	\end{itemize}
\end{theorem}

The method in \cite{KLMS} is based on hypercontractivity for global functions introduced by Keevash, Lifshitz, Long and Minzer \cite{Keevash-Lifshitz-Long-Minzer-2024}. This method has also led to progress on several extremal problems, including a solution to the forbidden cross-intersection problem for permutations \cite{KLS} in the essentially optimal range $n/\log n=\Omega(t)$.

For $\mathcal{F}\subseteq S_n$ and $\alpha,\beta\in\mathcal{S}_n$, we call $\alpha\mathcal{F}\beta:=\{\alpha\sigma\beta:\sigma\in\mathcal{F}\}$ a \emph{double translate} of $\mathcal{F}$. Two families  $\mathcal{F}_1,\mathcal{F}_2\subseteq\mathcal{S}_n$ are \emph{isomorphic}, denoted $\mathcal{F}_1\cong\mathcal{F}_2$, if they are the same up to double translate, that is, $\mathcal{F}_2=\alpha\mathcal{F}_1\beta$ for some $\alpha,\beta\in S_n$. Similarly, two pairs of families  $(\mathcal{F}_1,\mathcal{G}_1)$ and $(\mathcal{F}_2,\mathcal{G}_2)$ in $S_n$ are \emph{isomorphic} if $\alpha\mathcal{F}_1\beta=\mathcal{F}_2$ and $\alpha\mathcal{G}_1\beta=\mathcal{G}_2$ for some $\alpha,\beta\in S_n$. 

A central problem in this area is to determine the largest $t$-intersecting families in $S_n$ for all values of $n$ and $t$. Ellis, Friedgut and Pilpel \cite{Ellis-Friedgut-Pilpel-2011} conjectured an analogue of the complete intersection theorem of Ahlswede and Khachatrian. For $0\leq i\leq\lfloor(n-t)/2\rfloor$, define 
\[
\mathcal F_i(n,t)
=
\bigl\{\sigma\in S_n:
|\{j\in[t+2i]:\sigma(j)=j\}|\geq t+i
\bigr\}.
\]
These serve as natural analogues of the Frankl families \cite{Frankl-1976} for classical set systems. In particular, $\mathcal F_0(n,t)$ is a $t$-coset. The conjecture states that every $t$-intersecting family of maximum size is isomorphic to some $\mathcal F_i(n,t)$. Building on the spread approximation method, Kupavskii \cite{Kupavskii-2024} developed an iterative argument and proved that, for any fixed $\varepsilon>0$ and all sufficiently large $n$, the conjecture holds whenever $n>(1+\varepsilon)t$. Saengrungkongka \cite{Saengrungkongka} subsequently extended the range by further refining Kupavskii's techniques. In a recent landmark paper, Keller, Kupavskii, Lifshitz and Sheinfeld \cite{KKLS} established a complete intersection theorem for large permutation groups, that is, there exists a sufficiently large absolute constant $n_0$ such that the conjecture holds for all $n>n_0$ and every $1\leq t\leq n$. Their proof develops the spread approximation method into a general framework for $t$-intersection problems.

Let us turn to problems for non-trivial (cross) $t$-intersecting families. Define 
\begin{align*}
	\mathcal{H}(n,t):=\{\sigma\in S_n:
	\sigma(i)&=i\text{ for all }i\in[t],\\
	\sigma(j)&=j\text{ for some }j\in[t+1,n-1]\}\cup\{(i\,n):i\in[t]\}.
\end{align*}
Let $d_n:=n!\cdot\sum_{i=0}^n\frac{(-1)^i}{i!}$ be the $n$-th derangement number. By a routine counting argument,
$$|\mathcal{H}(n,t)|=(n-t)!-d_{n-t}-d_{n-t-1}+t.$$ In 2003, Cameron and Ku \cite{Cameron-Ku-2003} conjectured that $\mathcal H(n,1)$ is, up to isomorphism, the largest non-trivial $1$-intersecting family in $S_n$. The first step towards characterizing
the largest non-trivial $t$-intersecting families of
permutations was taken by Ellis \cite{Ellis-2012}, who proved this conjecture for all sufficiently large $n$ and thereby established a Hilton--Milner theorem for permutations. In \cite{Ellis-2011}, Ellis extended this result to all $t\geq1$, establishing a Hilton--Milner--Frankl theorem for permutations.
\begin{theorem}[\cite{Ellis-2011}]\label{thmsingleHMF}
For $n$ sufficiently large depending on $t$, if $\mathcal{F}\subseteq S_n$ is a non-trivial $t$-intersecting family, then $|\mathcal{F}|\leq(n-t)!-d_{n-t}-d_{n-t-1}+t$. Moreover, equality holds if and only if $\mathcal{F}\cong\mathcal{H}(n,t)$.
\end{theorem}
In the same paper, he also established two fundamental cross-intersection results of permutations for sufficiently large $n$ \footnote{The proof in \cite{Ellis-2011} appears to give a bound on $n$ that is
	doubly exponential in $t$. The main aim of that work was to establish
	exact stability results for $t$-intersecting families of permutations.}. Before presenting them, we introduce two structures of non-trivial cross $t$-intersecting families.\begin{construction}\label{construction2}
	Define $(\mathcal{A}(n,t),\mathcal{B}(n,t))$ as follows. 
	\begin{align*}
		\mathcal{A}(n,t)&:=\{\sigma\in S_n:\sigma(i)=i\;\mbox{for all}\;i\in[t],\;\sigma(j)=j\;\mbox{for some}\;j\in[t+1,n-1]\}.\\
		\mathcal{B}(n,t)&:=\{\sigma\in S_n:\sigma(i)=i\;\mbox{for all}\;i\in[t]\}\cup\{(i\ n):i\in[t]\}.
	\end{align*}
\end{construction}
\begin{construction}\label{construction}
Let $\alpha,\beta\in S_n$ with $\alpha(1)\neq1$ and $\beta(1)\neq1$. Further, for $t=1$, suppose $\alpha(j)=\beta(j)$ for some $j\in[n]$; for $t\geq2$, suppose $\alpha(i)=\beta(i)=i$ for $2\leq i\leq t$, and $\alpha(j)=\beta(j)>t$ for at least two $j\in[t+1,n]$. Define
	\begin{align*}
		\mathcal H(\alpha,\beta):
		={}&\bigl\{\sigma\in S_n:
		\sigma(i)=i\text{ for all }i\in[t],\
		|\{j\in[n]:\sigma(j)=\beta(j)\}|\geq t\bigr\}\\
		&\cup\{(1\,i)\alpha(1\,i):i\in[t]\}.
	\end{align*}
	The pair $(\mathcal{H}(\alpha,\beta),\mathcal{H}(\beta,\alpha))$ is cross $t$-intersecting. In particular, $\mathcal{H}(n,t)=\mathcal H((1\,n),(1\,n))$. 
\end{construction}

\begin{theorem}[\cite{Ellis-2011}]\label{thmEllis-cross}
For $n$ sufficiently large depending on $t$, if $(\mathcal{F},\mathcal{G})$ is a non-trivial pair of cross $t$-intersecting families in $S_n$, then the following hold.
\begin{itemize}
\item[\rm(i)]$|\mathcal{F}||\mathcal{G}|\leq((n-t)!-d_{n-t}-d_{n-t-1})((n-t)!+t)$, and equality holds if and only if $(\mathcal{F},\mathcal{G})$ or $(\mathcal{G},\mathcal{F})$ is isomorphic to the pair given in Construction \ref{construction2}.
\item[\rm(ii)]$\min\{|\mathcal{F}|,|\mathcal{G}|\}\leq(n-t)!-d_{n-t}-d_{n-t-1}+t$, and equality holds if and only if $(\mathcal{F},\mathcal{G})$ is isomorphic to a pair given in Construction \ref{construction}.
\end{itemize}
\end{theorem}	
\subsection{Main results}
Our first main result consisting of two intersection theorems in the linear range.

\begin{theorem}\label{thmpurelycombin}
Let $t\geq1$ and $n\geq400t$, and let $\mathcal{F},\mathcal{G}\subseteq S_n$ be cross $t$-intersecting families. Then the following hold.
\begin{itemize}
\item[\rm(i)]If $(\mathcal{F},\mathcal{G})$ is non-trivial, then $|\mathcal{F}||\mathcal{G}|\leq((n-t)!-d_{n-t}-d_{n-t-1})((n-t)!+t)$, and equality holds if and only if $(\mathcal{F},\mathcal{G})$ or $(\mathcal{G},\mathcal{F})$ is isomorphic to the pair given in Construction \ref{construction2}.
\item[\rm(ii)]If both $\mathcal{F}$ and $\mathcal{G}$ are non-trivial, then $|\mathcal{F}||\mathcal{G}|\leq((n-t)!-d_{n-t}-d_{n-t-1}+t)^2$, and equality holds if and only if $(\mathcal{F},\mathcal{G})$ is isomorphic to a pair given in Construction \ref{construction}.
\end{itemize}
\end{theorem}

As consequences of Theorem \ref{thmpurelycombin}, we obtain an explicit value
for the constant $c_0$ in Theorem \ref{thmKLMS}, and obtain a  Hilton--Milner--Frankl theorem of permutations for $n=\Omega(t)$ by taking $\mathcal F=\mathcal G$ in
Theorem \ref{thmpurelycombin} (ii).
\begin{corollary}\label{thmc0}
	The constant $c_0$ in Theorem \ref{thmKLMS} can be taken to be $400$.
\end{corollary}
\begin{corollary}
Theorem \ref{thmsingleHMF} holds for all $t\geq1$ and $n\geq400t$.
\end{corollary}

We note that Theorem \ref{thmpurelycombin} (i) establishes a sharp stability result for the product version of the Erd\H{o}s--Ko--Rado theorem for
permutations (Theorem \ref{thmKLMS} (i)), sharpening
Theorem \ref{thmKLMS} (ii)\footnote{The constant $0.75$ in
	Theorem \ref{thmKLMS} can be replaced by any fixed constant
	greater than $1-1/e$, with a corresponding adjustment of $c_0$.}.
Theorem \ref{thmpurelycombin} (ii) establishes a product version
of the Hilton--Milner--Frankl theorem for permutations and implies
Theorem \ref{thmEllis-cross} (ii). Indeed, it holds directly when $\mathcal{F}$ and $\mathcal{G}$ are non-trivial. If one of the families is trivial, then Lemma \ref{lemmaind} (ii) gives a strict inequality on  $\min\{|\mathcal{F}|,|\mathcal{G}|\}$. 

Let us note that the linear dependence on $t$ in these results are essentially optimal. Indeed, fix $1<c<1+(1-e^{-1})^{-1}$ and let $n=\lfloor ct\rfloor$. Consider the family $\mathcal F_1(n,t)$. We have
\[
\begin{aligned}
	|\mathcal F_1(n,t)|
	&=(t+2)(n-t-1)!-(t+1)(n-t-2)!\\
	&=\left(\frac{1}{c-1}+o(1)\right)(n-t)!>(n-t)!-d_{n-t}-d_{n-t-1}+t
\end{aligned}
\]
for all sufficiently large $t$, because 
$d_{n-t}+d_{n-t-1}=(e^{-1}+o(1))(n-t)!$ as $t\to\infty$. So Theorem \ref{thmsingleHMF} fails for $n=\lfloor ct\rfloor$ and all sufficiently large $t$. Similarly, taking
$\mathcal F=\mathcal G=\mathcal F_1(n,t)$ violates Theorem \ref{thmpurelycombin} {\rm(ii)}. For assertion {\rm(i)}, take $\mathcal F$ to consist of all
permutations fixing at least $t$ points of $[t+1]$, and
$\mathcal G$ to consist of all permutations fixing every
point of $[t+1]$. These families form a non-trivial pair of cross
$t$-intersecting families, and
\[
\begin{aligned}
	|\mathcal F||\mathcal G|
	&=\bigl((t+1)(n-t)!-t(n-t-1)!\bigr)(n-t-1)!\\
	&=\left(\frac{1}{c-1}+o(1)\right)((n-t)!)^2>\bigl((n-t)!-d_{n-t}-d_{n-t-1}\bigr)
	\bigl((n-t)!+t\bigr)
\end{aligned}
\]
for all sufficiently large $t$.

Our next main result is a stability theorem. Before presenting it, let us introduce some notation. For a family $\mathcal{F}$ of subsets of $[n]$, its \emph{maximum $s$-degree}, denoted $\Delta_s(\mathcal{F})$, is defined to be the maximum number of its members that contain a fixed $s$-subset. Formally,
\begin{equation}\label{equdefDelta}
	\Delta_s(\mathcal{F}):=\max_{S\in\binom{[n]}{s}}|\left\{F\in\mathcal{F}:S\subseteq F\right\}|.
\end{equation}
Its \emph{$t$-diversity} is defined as
\begin{equation*}
	\gamma_t(\mathcal{F}):=|\mathcal{F}|-\Delta_t(\mathcal{F}).
\end{equation*}
We view a permutation $\sigma\in S_n$ as the set
\begin{equation}\label{equsetinter}
\{(i,\sigma(i)):i\in[n]\}\subseteq[n]^2.
\end{equation}
Then $S_n$ can be regarded as a subfamily of $\binom{[n]^2}{n}$, and two families $\mathcal{F},\mathcal{G}\subseteq S_n$ are cross
$t$-intersecting if $|\sigma\cap\pi|\geq t$ for every
$\sigma\in\mathcal{F}$ and $\pi\in\mathcal{G}$. A subset of a permutation is called a \emph{partial permutation}. In this sense, we can define the $t$-diversity of a family of permutations. This parameter measures the distance from a family to its closest $t$-coset. More precisely, every $\mathcal{F}\subseteq S_n$ can be made a trivial family by deleting $\gamma_t(\mathcal{F})$ members. 

The key ingredient in Ellis's work \cite{Ellis-2011} is a stability result showing that, for sufficiently large $n$, every $t$-intersecting family of size at least a fixed positive proportion of $(n-t)!$ is almost contained in a $t$-coset; in terms of $t$-diversity, its $t$-diversity is $O((n-t-1)!)$. He further proposed the following sharp form of this result \cite{Ellis-2011,Ellis-2012}.
\begin{conjecture}[\cite{Ellis-2011}]\label{conjEllis}
	For $n$ sufficiently large depending on $t$, if $\mathcal{F}\subseteq S_n$ is $t$-intersecting, then there is a $t$-coset containing all but at most $t((n-t-1)!-(n-t-2)!)$ members of $\mathcal{F}$.
\end{conjecture}

Note that the upper bound is tight, as a direct counting argument gives 
$$
\gamma_t(\mathcal F_1(n,t))
=t\bigl((n-t-1)!-(n-t-2)!\bigr).$$
We remark that the case $t=1$ was proved by Wang and Xiao
\cite{Wang-Xiao} for $n\geq500$, although it was not mentioned there. Our next main result settles all remaining cases $t\geq2$ in a stronger form. 
\begin{theorem}\label{thmcrossdiv}
	Let $t\geq2$ and $n\geq 400t$. If $\mathcal{F},\mathcal{G}\subseteq S_n$ are cross $t$-intersecting, then 
	$$
	\min\{\gamma_t(\mathcal{F}),\gamma_t(\mathcal{G})\}\leq t((n-t-1)!-(n-t-2)!). 
	$$
	If equality holds, then $\mathcal{F},\mathcal{G}\subseteq\mathcal{H}$ for some $\mathcal{H}\cong\mathcal{F}_1(n,t)$.
\end{theorem}
For a single family, the argument also applies when \(t=1\), giving the following result.
\begin{theorem}\label{thmdiv}
	Conjecture \ref{conjEllis} is true whenever $t\geq1$ and $n\geq400t$.  Moreover, the bound $t((n-t-1)!-(n-t-2)!)$ is sharp only when  $\mathcal{F}$ is isomorphic to a subfamily of $\mathcal{F}_1(n,t)$.
\end{theorem}

Let us note that the linear dependence on $t$ in
Theorems \ref{thmcrossdiv} and \ref{thmdiv}
is also essentially optimal.
Indeed, let $n=\lfloor4t/3\rfloor$ and take
$\mathcal F=\mathcal G=\mathcal F_2(n,t)$.
A direct count gives
$|\mathcal F_2(n,t)|=(9/2+o(1))(n-t)!$ as $t\to\infty$,
whereas
$t((n-t-1)!-(n-t-2)!)=(3+o(1))(n-t)!$.
Since every $t$-coset has size $(n-t)!$, we obtain
\[
\gamma_t(\mathcal F_2(n,t))
\geq|\mathcal F_2(n,t)|-(n-t)!
=\left(\frac72+o(1)\right)(n-t)!
>\gamma_t(\mathcal F_1(n,t))
\]
for $n=\lfloor4t/3\rfloor$ and all sufficiently large $t$.

Our approach may apply to a variety of objects. As an application, we apply the method to the alternating group, proving a product version of the Hilton--Milner--Frankl theorem. 
\begin{theorem}\label{thmcrossAn}
	Let $t\geq1$ and $n\geq400t$, and let $\mathcal{F},\mathcal{G}\subseteq A_n$ be cross $t$-intersecting. If both $\mathcal{F}$ and $\mathcal{G}$ are non-trivial, then $|\mathcal{F}||\mathcal{G}|\leq0.25((n-t)!-d_{n-t}-d_{n-t-1}+(-1)^{n-t-1}+2t)^2$. Equality holds if and only if there exist $\rho,\pi\in S_n$ of the same parity and $\alpha,\beta\in A_n$ satisfying the conditions in
	Construction \ref{construction} such that $(\mathcal{F},\mathcal{G})=\big(\rho(\mathcal{H}(\alpha,\beta)\cap A_n)\pi,\rho(\mathcal{H}(\beta,\alpha)\cap A_n)\pi\big)$.
\end{theorem}
As a direct corollary, we obtain a Hilton--Milner--Frankl type theorem, which  has been established by Ellis \cite{Ellis-2011} for sufficiently large $n$ depending on $t$.
\begin{corollary}\label{thmAn}
	Let $t\geq1$ and $n\geq400t$, and let $\mathcal{F}\subseteq A_n$ be a non-trivial $t$-intersecting family. Then $|\mathcal{F}|\leq0.5((n-t)!-d_{n-t}-d_{n-t-1}+(-1)^{n-t-1})+t$.  Equality holds if and only if there exist $\rho,\pi\in S_n$ of the same parity and $\alpha\in A_n$ fixing all but the point $1$ of $[t]$ such that $\mathcal{F}=\rho(\mathcal{H}(\alpha,\alpha)\cap A_n)\pi$.
\end{corollary}
\subsection{Our approach}
Our proofs are purely combinatorial and use three main tools: spread approximation, the peeling procedure for a pair, and $t$-covers. We note that the iterative spread approximation argument developed in \cite{Kupavskii-2024,Saengrungkongka,KKLS} does not seem to extend directly to cross $t$-intersecting families, as the iterative step relies on intersection within a single family. The peeling procedure was introduced by Kupavskii
and Zakharov \cite{Kupavskii-Zakharov-2024} and generalized
to pairs of cross $t$-intersecting families by the present
authors \cite{Wen-Lv-2026+}. In our approach, the following three refinements are essential. 

\begin{itemize}
	\item \emph{Generating a spread approximation through $t$-covers.}\;A simple yet surprisingly useful observation is that the families under consideration admit $t$-covers. Therefore, before applying the iteration procedure generating spread approximations, we decompose the family into $t$-links via a $t$-cover, and then apply the original iteration technique separately to the corresponding $t$-links (Theorem \ref{thmspreadapp}). The resulting spread approximation is then lifted and collected into a single family. Although this preliminary step is elementary, its effect on the remainder is substantial. Roughly speaking, for all sufficiently large $t$, and with the parameters used here, the resulting remainder is at most $(8/t)^t$ times the one obtained by applying the method directly. This is essential to settle the problem for $n\geq400t$.

	\item \emph{Spread approximation for a pair.}\;Applying spread approximation to the two families separately does not ensure that the resulting families remain cross $t$-intersecting. In addition, the peeling procedure for a pair does not give useful estimates when applied directly to the original permutation families. This is because the usual set interpretation (given in (\ref{equsetinter})) of permutations produces a family that is \emph{not} sufficiently spread, leading to bad estimates. To overcome these difficulties, we prove a cross version of spread approximation (Theorem \ref{thmcrossspreadapp}). More precisely, given a pair $(\mathcal{F},\mathcal{G})$ of cross $t$-intersecting families, there exists a pair $(\mathcal{S},\mathcal{T})$ of families consisting of sets of much smaller size such that all but a small number of members of $\mathcal{F}$ (of $\mathcal{G}$) contain some set from $\mathcal{S}$ (from $\mathcal{T}$). Moreover, when $n$ is large compared with $t$, we can require that $\mathcal{S}$ and $\mathcal{T}$ are also cross $t$-intersecting and that both remainders are small enough for our later estimates. We then apply the peeling procedure to the pair $(\mathcal{S},\mathcal{T})$ to characterize the original pair $(\mathcal{F},\mathcal{G})$. This provides the link between spread approximation and the peeling procedure for a pair. 
	
	\item \emph{A decomposition of the Cartesian product.}
	To estimate $|\mathcal F||\mathcal G|$, we use the following decomposition given in Lemma \ref{lemmadecomp}.
	$$
	\mathcal F[\mathcal S]\times\mathcal G[\mathcal T]
	\subseteq
	\bigcup_{i=0}^{N}
	\bigl(\mathcal F[\mathcal X_i]\times\mathcal G[\mathcal T_i]\bigr)
	\cup
	\bigcup_{i=0}^{N-1}
	\bigl(\mathcal F[\mathcal S_i]\times\mathcal G[\mathcal Y_i]\bigr).
	$$
	This allows us to make effective use of an important ingredient
	in the $t$-cover method (Lemma \ref{lemmatcoverkey-perm}). More precisely, the key to using this lemma is to find a collection of \(t\)-covers whose members are relatively small and whose $t$-covering number is large. For a pair $(\mathcal{F}[\mathcal{X}_i], \mathcal{G}[\mathcal{T}_i])$, the members of $\mathcal X_i$ have size $O(t\ln(n/t))$, which is sufficiently small compared with $n$. Then we use $\mathcal X_i$ as a collection of $t$-covers to estimate $|\mathcal G[\mathcal T_i]|$, where the resulting bound depends on the $t$-covering number of $\mathcal X_i$. We retain this dependence when combining it with the estimate
	for $|\mathcal F[\mathcal X_i]|$. The same argument applies to every pair
	$(\mathcal F[\mathcal S_i],\mathcal G[\mathcal Y_i])$.
	Applying the estimates separately to the families 
	gives the desired bound only for 
	$n/\ln n=\Omega(t^2)$. Using this decomposition, we obtain the product estimate in Lemma \ref{lemmagroupproduct} under the linear condition \(n\geq400t\).
\end{itemize}
\subsection{Organization}
The rest of this paper is organized as follows. In Section \ref{sec2}, we  recall the peeling procedure for a pair and prove two theorems on spread approximation (Theorems \ref{thmspreadapp} and \ref{thmcrossspreadapp}). In Section \ref{sec3}, we derive estimates using the peeling procedure and combine them with the
approximation theorems to prove Theorems \ref{thmpurelycombin}, \ref{thmcrossdiv}
and \ref{thmdiv}. In Section \ref{secalternating}, we adapt the approach to the alternating group, and prove Theorem  \ref{thmcrossAn}.

\section{Preliminaries and spread approximation}\label{sec2}
In this section, we first fix notation and recall the peeling procedure for a pair of cross $t$-intersecting families. We then prove a spread approximation theorem for families admitting $t$-covers and apply it to cross $t$-intersecting families of permutations.

For families $\mathcal{A}$ and $\mathcal{S}$ of subsets of a set $\Omega$ and a subset $X$ of $\Omega$, we write
\begin{align*}
	\mathcal{A}[X]&:=\{A:A\in\mathcal{A},X\subseteq A\}.\\
	\mathcal{A}(X)&:=\{A\setminus X:A\in\mathcal{A},X\subseteq A\}.\\
	\mathcal{A}[\mathcal{S}]&:=\left\{A:A\in\mathcal{A}, S\subseteq A\;\mbox{for some}\;S\in\mathcal{S}\right\}.
\end{align*}
Note that $\mathcal{A}(X)$ has the same size as $\mathcal{A}[X]$, while it is a family of subsets of $\Omega\setminus X$. For $i\geq0$, the \emph{maximum $i$-degree} of $\mathcal{A}$ is defined to be
\begin{equation*}
	\Delta_i(\mathcal{A}):=\max\{|\mathcal{A}[S]|:S\subseteq\Omega,\;|S|=i\}.
\end{equation*}
The family $\mathcal{A}$ is said to be \emph{$r$-spread} if 
$$\Delta_s(\mathcal{A})\leq r^{-s}|\mathcal{A}|\;\;\mbox{for all}\;\;s\geq0.$$
It is \emph{weakly $(r,t)$-spread} if 
$$\Delta_{t+s}(\mathcal{A})\leq r^{-s}\Delta_{t}(\mathcal{A})\;\;\mbox{for all}\;\;s\geq0.$$
Let us collect a useful estimate by Saengrungkongka \cite{Saengrungkongka}.
\begin{equation}\label{equfractorial}
	\frac{m!}{a!}\geq\left(\frac{m}{e}\right)^{m-a}\;\;\mbox{for}\;\;m\geq a.
\end{equation} 
The following lemma is readily verified.
\begin{lemma}\label{lemmasnspread}
	The family $S_n$ is $r$-spread for $r\leq n/e$, and is weakly $(r,q)$-spread for $r\leq(n-q)/e$. 
\end{lemma}

For a family $\mathcal{A}$ of subsets of $\Omega$, a set $T\subseteq\Omega$ is a \emph{$t$-cover} of $\mathcal A$ if $|T\cap A|\geq t$ for every $A\in\mathcal A$. If such a subset exists, then the \emph{$t$-covering number} of $\mathcal{A}$, denoted $\tau_t(\mathcal A)$, is defined to be 
the minimum size of a $t$-cover of $\mathcal A$.
A $t$-cover is \emph{minimal} if none of its proper subsets
is a $t$-cover. 

Let $\mathcal F\subseteq S_n$ be a family of permutations. We define its $t$-covers and $t$-covering number as above by regarding $\mathcal F$ as a subfamily of $\binom{[n]^2}{n}$ through the set interpretation \eqref{equsetinter}. In terms of permutations, a $t$-cover can equivalently be described as a collection of distinct assignments $i_1\mapsto j_1,\ldots,i_m\mapsto j_m$, where $i_s,j_s\in[n]$, such that every $\sigma\in\mathcal F$ satisfies $\sigma(i_s)=j_s$ for at least $t$ indices $s\in[m]$. Then the $t$-covering number $\tau_t(\mathcal F)$ is the minimum number of such assignments.

Next, we recall the peeling procedure for a pair from \cite{Wen-Lv-2026+}. The procedure repeatedly uses the following lemma to replace a pair of cross $t$-intersecting families by a fingerprint. The resulting sequences of fingerprints encode information about the original families that can be used both to estimate their sizes and to characterize their structure.

\begin{lemma}[\cite{Wen-Lv-2026+}]\label{lemmafingerprintdef}
	Let $\Omega$ be a finite set and $(\mathcal{S},\mathcal{T})$ be a pair of cross $t$-intersecting families of subsets of $\Omega$. Then there is a pair of cross $t$-intersecting families $(\mathcal{S}^*,\mathcal{T}^*)$, called a \emph{fingerprint} of $(\mathcal{S},\mathcal{T})$, such that
	\begin{itemize}
		\item[\rm(i)]Every member of $\mathcal{S}^*$ {\rm(}of $\mathcal{T}^*${\rm)} is contained in some of $\mathcal{S}$ {\rm(}of $\mathcal{T}${\rm)};
		\item[\rm(ii)]Every member of $\mathcal{S}$ {\rm(}of $\mathcal{T}${\rm)} contains some of $\mathcal{S}^*$ {\rm(}of $\mathcal{T}^*${\rm)};
		\item[\rm(iii)]Every member of $\mathcal{S}^*$ {\rm(}of $\mathcal{T}^*${\rm)} is a minimal $t$-cover of $\mathcal{T}^*$ {\rm(}of $\mathcal{S}^*${\rm)}.
	\end{itemize}
\end{lemma}

\begin{algorithm}[\cite{Wen-Lv-2026+}, Peeling procedure for a pair]\label{algo}
Suppose that a finite set $\Omega$, integers $q$ and $t$ with $q\geq t$, a pair of cross $t$-intersecting families $(\mathcal{S},\mathcal{T})$ with  $\mathcal{S},\mathcal{T}\subseteq\binom{\Omega}{\leq q}$, and a fingerprint $(\mathcal{S}_0,\mathcal{T}_0)$ of $(\mathcal{S},\mathcal{T})$ are given. Set $N=0$. For $i=0,1,\ldots,q-t$, do the following:

\begin{itemize}
	\item[\rm(a)]Set  $\mathcal{X}_{i}=\mathcal{S}_{i}\cap\binom{\Omega}{q-i}$ and $\mathcal{Y}_{i}=\mathcal{T}_{i}\cap\binom{\Omega}{q-i}$. If $\mathcal{S}_i=\mathcal{X}_i$ or $\mathcal{T}_i=\mathcal{Y}_i$,
	terminate the procedure and set $N=i$.
	\item[\rm(b)]Find a fingerprint of the pair  $(\mathcal{S}_{i}\setminus\mathcal{X}_{i},\mathcal{T}_{i}\setminus\mathcal{Y}_{i})$, and denote it by $(\mathcal{S}_{i+1},\mathcal{T}_{i+1})$.
\end{itemize}
 Output $N$ and $(\mathcal{S}_{i},\mathcal{T}_{i},\mathcal{X}_{i},\mathcal{Y}_i)$ for $i\leq N$.
\end{algorithm}
\begin{lemma}[\cite{Wen-Lv-2026+}]\label{lemmafingerprintproperty}
	The following hold for $i=0,1,\ldots,N$.
	\begin{itemize}
		\item[\rm(i)] $\mathcal{S}_i, \mathcal{T}_i\subseteq\binom{\Omega}{\leq(q-i)}$ and $\mathcal{X}_{i}, \mathcal{Y}_i\subseteq\binom{\Omega}{q-i}$, and all of them are antichains.
		\item[\rm(ii)]$\mathcal{S}[\mathcal{S}_{i-1}]\subseteq\mathcal{S}[\mathcal{X}_{i-1}]\cup\mathcal{S}[\mathcal{S}_i]$ and $\mathcal{T}[\mathcal{T}_{i-1}]\subseteq\mathcal{T}[\mathcal{Y}_{i-1}]\cup\mathcal{T}[\mathcal{T}_i]$ for $i\geq1$.
		\item[\rm(iii)]$\mathcal{S}\subseteq\left(\bigcup_{j=0}^{i-1}\mathcal{S}[\mathcal{X}_j]\right)\cup\mathcal{S}[\mathcal{S}_i]$ and $\mathcal{T}\subseteq\left(\bigcup_{j=0}^{i-1}\mathcal{T}[\mathcal{Y}_j]\right)\cup\mathcal{T}[\mathcal{T}_i]$ for $i\geq1$.
		\item[\rm(iv)]If $\mathcal{H}$ is a $(q-a)$-uniform subfamily of $\mathcal{S}_i$ or $\mathcal{T}_i$, then $\Delta_t(\mathcal{H})\leq(q-t-i+1)^{q-t-a}$.
		\item[\rm(v)]$\max\{|\mathcal{X}_i|,|\mathcal{Y}_i|\}\leq\binom{q-i}{t}(q-t-i+1)^{q-t-i}$.
	\end{itemize}
\end{lemma}

We next prove a spread approximation theorem in which the remainder bound depends on the $t$-covering number. The proof applies the standard argument (see, e.g., \cite{Kupavskii-2026}) separately to the links $\mathcal F(X)$, where $X$ ranges over the $t$-subsets of a $t$-cover of $\mathcal F$. Then a desired approximation of $\mathcal{F}$ is obtained by lifting back each set from one of those links with the corresponding center.

\begin{theorem}\label{thmspreadapp}
	Let $1\leq t\leq q$, $m\geq t$ and $r_0\geq r\geq1$. Suppose $\Omega$ is a finite set and  $\mathcal{A}$ is a family of subsets of $\Omega$ such that  $\mathcal{A}(X)$ is weakly $(r_0,q-t+1)$-spread for every $t$-subset $X$ of $\Omega$. Then for each  $\mathcal{F}\subseteq\mathcal{A}$ with $\tau_t(\mathcal{F})\leq m$, there exists a family  $\mathcal{S}\subseteq\binom{\Omega}{\leq q}$, called a \emph{spread approximation} of $\mathcal{F}$, such that
	\begin{itemize}
		\item[\rm(i)] for every $S\in\mathcal S$, there exists
		$\mathcal{F}_S\subseteq\mathcal{F}$ such that
		$\mathcal{F}_S(S)$ is non-empty and $r$-spread;
		\item[\rm(ii)]$|\mathcal{F}\setminus\mathcal{F}[\mathcal{S}]|
		\leq\binom{m}{t}r^{q-t+1}\Delta_{q+1}(\mathcal{A})$.
	\end{itemize}
\end{theorem}
\begin{proof}
	Choose a $t$-cover $L$ of $\mathcal{F}$ with $|L|=\tau_t(\mathcal{F})$. Let $\mathcal{X}$ be the collection of $X\in\binom{L}{t}$ with $\mathcal{F}[X]\neq\emptyset$. Then 
	\begin{equation}\label{equthmspreadapp1}
		\mathcal{F}=\bigcup_{X\in\mathcal{X}}\mathcal{F}[X].
	\end{equation}
	We construct a family $\mathcal{S}_X\subseteq\binom{\Omega}{\leq q}$ for each $X\in\mathcal{X}$ as follows.
	
	Fix an $X\in\mathcal{X}$, and set  $\mathcal{C}_1:=\mathcal{F}(X)$. For $i\geq1$, do the following: 
	\begin{itemize}
		\item[\rm(a)]If $\mathcal{C}_i=\emptyset$, terminate. Otherwise, find an inclusion-maximal subset $W_i\subseteq\Omega\setminus X$ with $|\mathcal{C}_i[W_i]|\geq r^{-|W_i|}|\mathcal{C}_i|$.
		\item[\rm(b)]If $|W_i|>q-t$, terminate. Otherwise, record $W_i$ and put  $\mathcal{C}_{i+1}:=\mathcal{C}_i\setminus\mathcal{C}_i[W_i]$.\vspace{1em}
	\end{itemize}
Whenever $\mathcal C_{i+1}$ is defined, it is a proper
subfamily of $\mathcal C_i$, so there are at most
$|\mathcal F(X)|$ such steps.
Let $j$ be the index at which the procedure terminates, and put
	$$\mathcal{W}_X:=\{W_1,\ldots,W_{j-1}\}\;\;\mbox{and}\;\;\mathcal{S}_X:=\{W_1\cup X,\ldots,W_{j-1}\cup X\}.$$ For every $W_i\in \mathcal{W}_X$ and every non-empty
	$T\subseteq\Omega\setminus(X\cup W_i)$, the maximality of $W_i$ gives
	$$
	|(\mathcal C_i(W_i))[T]|=|\mathcal C_i[W_i\cup T]|<r^{-|W_i\cup T|}|\mathcal C_i|\leq r^{-|T|}|\mathcal C_i[W_i]|.
	$$
	Thus $\mathcal C_i(W_i)$ is $r$-spread. Lifting back by $X$, we define
	\begin{equation}\label{equthmspreadapp2}
		\mathcal{F}_{X,W_i}:=\{F\in\mathcal{F}[X]:F\setminus X\in\mathcal{C}_i[W_i]\}.
	\end{equation}
	Then $\mathcal F_{X,W_i}(X\cup W_i)=\mathcal C_i(W_i)$ is non-empty and $r$-spread. 
	
	 Next, let us estimate the part left uncovered by $\mathcal{W}_X$. It is clear that $\mathcal{C}_1=\mathcal{C}_1[\mathcal{W}_X]$ when the procedure terminates at $\mathcal{C}_j=\emptyset$. Otherwise, $|W_j|\geq q-t+1$, and hence
	\begin{align}
		|\mathcal{F}[X]\setminus\mathcal{F}[\mathcal{S}_X]|&=|\mathcal{C}_1\setminus\mathcal{C}_1[\mathcal{W}_X]|=|\mathcal{C}_j|\nonumber\\
		&\leq r^{|W_j|}|\mathcal{C}_j[W_j]|\leq r^{|W_j|}\Delta_{|W_j|}(\mathcal{A}(X))\nonumber\\
		&\leq r^{q-t+1}\Delta_{q-t+1}(\mathcal{A}(X))\leq r^{q-t+1}\Delta_{q+1}(\mathcal{A}).\label{equthmspreadapp3}
	\end{align}
	The third inequality follows from $r_0\geq r$ and the weak
	$(r_0,q-t+1)$-spreadness of $\mathcal{A}(X)$, and the last one follows from
	$|\mathcal{A}(X)[W]|=|\mathcal{A}[X\cup W]|$ whenever $|W|=q-t+1$.
	
	Let us verify that 
	$$\mathcal{S}:=\bigcup_{X\in\mathcal{X}}\mathcal{S}_X$$ 
	has the required properties. Of course $\mathcal{S}\subseteq\binom{\Omega}{\leq q}$ since this holds for every  $\mathcal{S}_X$. Next, for each $S\in\mathcal{S}$, there exist $X\in\mathcal{X}$ and $W_i\in\mathcal{W}_X$ with $S=X\cup W_i$. Then we may put $\mathcal{F}_S:=\mathcal{F}_{X,W_i}$, as defined in (\ref{equthmspreadapp2}). As shown above,  $\mathcal{F}_S(S)\neq\emptyset$ and $\mathcal{F}_S(S)$ is $r$-spread. It remains to prove (ii). We observe that
	\begin{equation}\label{equthmspreadapp4}
		\mathcal{F}\setminus\mathcal{F}[\mathcal{S}]\subseteq\bigcup_{X\in\mathcal{X}}(\mathcal{F}[X]\setminus\mathcal{F}[\mathcal{S}_X]).
	\end{equation}
	To see this, let $F\in\mathcal{F}\setminus\mathcal{F}[\mathcal{S}]$. Then by (\ref{equthmspreadapp1}), $F\in\mathcal{F}[X]$ for some $X\in\mathcal{X}$. Moreover, $F$ cannot contain $X\cup W$ for
	any $W\in\mathcal W_X$, since $X\cup W\in\mathcal S$. Thus
	$F\notin\mathcal F[\mathcal S_X]$, and this proves (\ref{equthmspreadapp4}). Finally, using $|\mathcal{X}|\leq\binom{\tau_t(\mathcal{F})}{t}\leq\binom{m}{t}$ and the estimate (\ref{equthmspreadapp3}) with $X$ ranging over $\mathcal{X}$, we derive
	$$|\mathcal{F}\setminus\mathcal{F}[\mathcal{S}]|\leq\binom{m}{t}r^{q-t+1}\Delta_{q+1}(\mathcal{A}).$$
	This finishes the proof.
\end{proof}
We now apply Theorem \ref{thmspreadapp} to a pair of cross $t$-intersecting
families of permutations.
To show that the resulting approximations are cross
$t$-intersecting, we use the following form of the Spread Lemma. For $0\leq p\leq1$, a \emph{$p$-random subset} of a finite
set $\Omega$ is obtained by including each element of
$\Omega$ independently with probability $p$.

\begin{theorem}[Spread Lemma, \cite{Tao,Stoeckl}]\label{spreadlemma}
	If $\mathcal{F}\subseteq\binom{[n]}{\leq k}$ is $r$-spread and $W$ is an $(m\delta)$-random subset of $[n]$, then 
	$$\mathbb{P}[F\subseteq W\;\mbox{for some}\; F\in\mathcal{F}]\geq1-\left(\frac{1+H(\delta)}{\log_2(r\delta)}\right)^mk,$$
	where $H(\delta)=-\delta\log_2\delta
	-(1-\delta)\log_2(1-\delta).$
\end{theorem}
In the breakthrough paper \cite{sunflower}, a key observation concerning the intersection problem from \cite[Theorem 4.2]{sunflower} is that a sufficiently spread (a `pseudorandom' property) family does contain two disjoint members (a specific one), and consequently cannot be $1$-intersecting. A similar observation also plays a crucial role in the spread approximation method. Combining Theorem \ref{thmspreadapp} with the Spread Lemma, we obtain the following approximation theorem. Its conclusion will allow us to apply the peeling procedure to the approximating pair in Section~\ref{sec3}.
 By a slight abuse of notation, the empty family is cross $t$-intersecting with every family.
\begin{theorem}\label{thmcrossspreadapp}
Let $t\geq1$ and $n\geq 400t$. Let 
	$\mathcal F,\mathcal G\subseteq S_n$ be cross
	$t$-intersecting. Then there exist families $\mathcal{S}$ and $\mathcal{T}$, each consisting of partial
	permutations of size at most $t+\left\lceil3\ln\binom{n}{t}\right\rceil$, such that
	\begin{itemize}
		\item[\rm(i)] $\max\big\{|\mathcal F\setminus\mathcal F[\mathcal{S}]|,\,|\mathcal{G}\setminus\mathcal G[\mathcal T]|\big\}\leq\binom{n}{t}^{-2}(n-t)!$, and 
		\item[\rm(ii)] $\mathcal S$ and $\mathcal T$ are cross
		$t$-intersecting.
	\end{itemize}
\end{theorem}

\begin{proof}
Put $
q:=t+\left\lceil3\ln\binom{n}{t}\right\rceil$ and $r=r_0=(n-q-1)/e$. We first verify that 
$$n\geq16q\;\;\mbox{and}\;\;r-q>17\ln(2n).$$ 
Indeed, write $f(x)=5+3\ln x$. Since $\binom{n}{t}\leq(en/t)^t$, we have
$q\leq4t+3t\ln(n/t)+1\leq tf(n/t)$. By a routine calculation, we have $x\geq16f(x)$ for $x\geq400$. So we obtain $n\geq16q$ by replacing $x$ with $n/t$. Next, $\frac{r-q}{t}
\geq\frac{n/t-1}{e}-(1+\frac{1}{e})f(n/t)$, and it is also routine to check that $g(x):=\frac{x-1}{e}-(1+\frac{1}{e})f(x)-17\ln (2x)$ satisfies $g(x)\geq g(400)>0$ for $x\geq400$. Hence $r-q>17t\ln(2n/t)\geq17\ln(2n)$, as required.
	
Since $\mathcal{F}$ and $\mathcal{G}$ are
	cross $t$-intersecting, each member of $\mathcal{G}$ is a $t$-cover of $\mathcal{F}$, and vice versa. Hence both families have $t$-covering number at most $n$. For every $t$-partial
	permutation $X$, the link $S_n(X)$ is naturally identified with $S_{n-t}$. Hence, by Lemma \ref{lemmasnspread}, it is weakly
	$(r_0,q-t+1)$-spread.	We may therefore apply Theorem \ref{thmspreadapp} with $\Omega=[n]^2$, $\mathcal{A}=S_n$ and $m=n$. By applying the theorem to $\mathcal{F}$ and $\mathcal{G}$, respectively, we get two families $\mathcal{S},\mathcal{T}\subseteq\binom{\Omega}{\leq q}$ with
	\begin{itemize}
		\item[\rm(a)] for all $S\in\mathcal{S}$ ($T\in\mathcal{T}$), there exists $\mathcal{F}_S\subseteq\mathcal{F}$ ($\mathcal{G}_T\subseteq\mathcal{G}$) such that $\mathcal{F}_S(S)$ ($\mathcal{G}_T(T)$) is non-empty and $r$-spread;
		\item[\rm(b)]$\max\{|\mathcal{F}\setminus\mathcal{F}[\mathcal{S}]|,|\mathcal{G}\setminus\mathcal{G}[\mathcal{T}]|\}\leq\binom{n}{t}r^{q-t+1}\Delta_{q+1}(\mathcal{A})=\binom{n}{t}r^{q-t+1}(n-q-1)!$.
	\end{itemize}
	Since each link $\mathcal{F}_S(S)$ or $\mathcal{G}_T(T)$ is non-empty, the families $\mathcal{S}$ and $\mathcal{T}$ consist of partial permutations. To prove (i), it suffices to estimate the bound given in (b). By $\binom{n}{t}\leq(en/t)^t$ and $q-t=\left\lceil3\ln\binom{n}{t}\right\rceil$, we have $e^{q-t}\geq\binom{n}{t}^3$. It follows that
	\begin{align*}
r^{q-t+1}(n-q-1)!/(n-t)!&=\prod_{i=1}^{q-t+1}\frac{n-q-1}{e(n-q-1+i)}<e^{-(q-t+1)}\leq e^{-1}\binom{n}{t}^{-3}.
	\end{align*}
So the remainder $\binom{n}{t}r^{q-t+1}(n-q-1)!$ is less than $\binom{n}{t}^{-2}(n-t)!$.
	
It remains to prove that $\mathcal S$ and $\mathcal T$ are cross
	$t$-intersecting. We prove a key claim.
\begin{claim}\label{claimcrossspreadapp1}
	Any two $(r-q)$-spread families in $\binom{\Omega}{\leq n}$ are not cross-intersecting.
\end{claim}
\begin{proof}
	Suppose that $\mathcal{G}_1,\mathcal{G}_2\subseteq\binom{\Omega}{\leq n}$ are $(r-q)$-spread. Pick a random $2$-partition $\Omega=W_1\cup W_2$ by assigning each element of $\Omega$ independently and uniformly to $W_1$ or $W_2$. Now every $W_i$ has the same distribution as a $(1/2)$-random subset of $[n]^2$. We are going to apply the Spread Lemma with  $m=\left\lceil2\ln(2n)\right\rceil$ and $\delta=(2m)^{-1}$. Since $n\geq400t$, we have $m\geq14$ and $H(\delta)\leq H(1/28)<0.23$. It is routine to check that the function $h(x):=2x\cdot2^{1.23e^{1/x}}$ is increasing on $[2,2.15]$. Moreover, $h(2.15)<16.712<17$. Then 
	\[
	\delta^{-1}\cdot2^{(1+H(\delta))(2n)^{1/m}}<2m\cdot2^{1.23e^{\ln(2n)/m}}=(\ln(2n))h(m/\ln(2n))<17\ln(2n)<r-q,
	\]
where the second inequality follows from the fact that $m/\ln(2n)\leq2+1/(\ln(2n))<2.15$ for $n\geq400$. Hence
	\[
	\left(\frac{1+H(\delta)}{\log_2((r-q)\delta)}\right)^m\cdot n
	<ne^{-\ln(2n)}=\frac1{2}.
	\]
	Then Theorem \ref{spreadlemma} gives
	$$\mathbb{P}[G\subseteq W_i\;\mbox{for some}\;G\in\mathcal{G}_i]\geq1-\left(\frac{1+H(\delta)}{\log_2((r-q)\delta)}\right)^m\cdot n>1-\left(\frac{1}{2n}\right)\cdot n=\frac{1}{2}$$
	for $i=1,2$. Therefore there is a specified $2$-partition $W_1, W_2$ together with $G_1\in\mathcal{G}_1,G_2\in\mathcal{G}_2$ such that $G_i\subseteq W_i$ for $i=1,2$, and certainly $G_1\cap G_2=\emptyset$.
\end{proof}
Assume to the contrary that $\mathcal{S}$ and $\mathcal{T}$ are not cross $t$-intersecting. Then $|A_1\cap A_2|<t$ for some $A_1\in\mathcal{S}$ and $A_2\in\mathcal{T}$. Set $\mathcal{H}_1=\mathcal{F}_{A_1}(A_1)$ and $\mathcal{H}_2=\mathcal{G}_{A_2}(A_2)$. By (a), both $\mathcal{H}_1$ and $\mathcal{H}_2$ are $r$-spread. Put
$$\mathcal{H}_1'=\{H\in\mathcal{H}_1: H\cap(A_2\setminus A_1)=\emptyset\}\;\;\mbox{and}\;\;\mathcal{H}_2'=\{H\in\mathcal{H}_2: H\cap(A_1\setminus A_2)=\emptyset\}.$$
From $|A_2|\leq q$ and the $r$-spreadness of $\mathcal{H}_1$,
$$
|\mathcal{H}_1\setminus\mathcal{H}_1'|\leq
\sum_{x\in A_2\setminus A_1}|\mathcal{H}_1[\{x\}]|
\leq|A_2|\cdot(r^{-1}|\mathcal{H}_1|)\leq\frac{q}{r}\cdot|\mathcal{H}_1|.$$
Hence $|\mathcal{H}_1'|\geq(1-\frac{q}{r})|\mathcal{H}_1|$. Then $\mathcal{H}_1'$ is $(r-q)$-spread.  More precisely, for every non-empty  subset $T$ of $\Omega\setminus A_1$, we have
$$|\mathcal{H}_1'[T]|\leq|\mathcal{H}_1[T]|\leq r^{-|T|}\cdot|\mathcal{H}_1|\leq(1-q/r)^{-1}r^{-|T|}\cdot|\mathcal{H}_1'|\leq(r-q)^{-|T|}|\mathcal{H}_1'|.$$
Hence $\mathcal{H}_1'$ is $(r-q)$-spread. Similarly, $\mathcal{H}_2'$ is $(r-q)$-spread as well. By Claim \ref{claimcrossspreadapp1}, there are $H_1\in\mathcal{H}_1'$ and $H_2\in\mathcal{H}_2'$ such that $H_1\cap H_2=\emptyset$. However, for $F:=H_1\cup A_1\in\mathcal{F}$ and  $G:=H_2\cup A_2\in\mathcal{G}$, 
$$|F\cap G|=|A_1\cap A_2|<t,$$
contrary to the assumption that $\mathcal{F}$ and $\mathcal{G}$ are cross  $t$-intersecting. This proves (ii).
\end{proof}

\section{Estimating families via fingerprints}\label{sec3}
\subsection{Counting lemmas}
For simplicity, let us define two expressions. First, for $1\leq a\leq m$, define
\begin{equation}\label{equfund}
d(m,a):=\sum_{i=0}^{a}(-1)^i(m-i)!\binom{a}{i}.
\end{equation}
By the inclusion-exclusion principle, $d(m,a)$ counts the number
of permutations of $[m]$ avoiding a fixed partial permutation of
size $a$. In particular, $d(m,m)=d_m$ is the $m$-th  derangement number. We also define
\begin{equation}\label{equsizeh}
h(n,t):=(n-t)!-d_{n-t}-d_{n-t-1}.
\end{equation}
Let us record several facts about these two expressions.
They follow readily from \eqref{equfund}, their combinatorial interpretation, and 
the Bonferroni inequalities.
\begin{lemma}\label{lemmadrgmt}
For $m\geq3$ and $1\leq a\leq m$, the following hold.
\begin{itemize}
\item[\rm(i)]$d(m,1)\geq d(m,2)\geq\cdots\geq d(m,m)=d_m$.
\item[\rm(ii)]$0.33m!\leq d_m\leq0.38m!$ and $0.6(m-t)!\leq h(m,t)\leq0.66(m-t)!$ for $m\geq t+20$.
\end{itemize}
\end{lemma}
\begin{lemma}\label{lemmaind}
Let $t\geq1$, $n\geq 400t$ and $q=t+\left\lceil3\ln\binom{n}{t}\right\rceil$. Let $\mathcal{D}\subseteq S_n$, and let $X$ and $L$ be partial permutations of $[n]$ of size $t$ and $\ell$, respectively. Suppose that $L$ is a $t$-cover of $\mathcal{D}$ and $|X\cap L|=:s<t$.
\begin{itemize}
\item[\rm(i)]If $\ell\leq q$, then $|\mathcal{D}[X]|\leq q(n-t-1)!$.
\item[\rm(ii)]If $\ell=n$, namely, $L\in S_n$, then $|\mathcal{D}[X]|\leq h(n,t)$, with equality precisely if $s=t-1$ and  $\mathcal{D}[X]$ consists of all $F\in S_n$ with $X\subseteq F$ and $F\cap(L\setminus X)\neq\emptyset$. Moreover, $|\mathcal{D}[X]|\leq0.5(n-t)!$ for $s\leq t-2$.

\item[\rm(iii)]Suppose $\ell=n, s=t-1$, and further $|X\cap K|=t-1$ for another $t$-cover $K\in S_n$ of $\mathcal{D}$ with $L(j)\neq K(j)$ for some $j\in[n]\setminus(X^{(1)}\cup\{k\})$, where $X^{(1)}=\{x:(x,y)\in X\}$ and $k$ is the index in $[n]\setminus X^{(1)}$ with  $X\setminus L=\{(X^{-1}(L(k)),L(k))\}$. Then $|\mathcal{D}[X]|\leq h(n,t)-d_{n-t-1}$.
\end{itemize}
\end{lemma}
\begin{proof}
To begin with, we note that $n\geq400t$ gives $n\geq2q$. Write $s=|X\cap L|$. Every member of $\mathcal{D}[X]$ shares at least $t$ elements with $L$, and thus intersects $L\setminus X$ in at least $t-s$ elements. So we estimate via the decomposition $$\mathcal{D}[X]=\bigcup_{H\in\mathcal{H}}\mathcal{D}[X\cup H],$$
where $\mathcal{H}$ is the collection of partial permutations of size $t-s$ of $L\setminus X$.

(i)\;When $\ell\leq q$, a routine union bound yields
\begin{align*}
	|\mathcal{D}[X]|&\leq\binom{\ell-s}{t-s}(n-(2t-s))!\leq(\ell-t+1)(n-t-1)!\leq q(n-t-1)!.
\end{align*}
To see the second inequality, just note that the function $\binom{\ell-x}{t-x}(n-(2t-x))!$ is increasing as $x\in\{0,1,\ldots,t-1\}$ increases, because $n\geq2q$.

(ii)-(iii)\;If $s\leq t-2$, then every permutation from $\mathcal{D}[X]$ must agree with $L\setminus X$ on at least two points, and hence 
$$|\mathcal{D}[X]|\leq\binom{n-t}{2}(n-t-2)!=0.5(n-t)!.$$

In what follows, we suppose $s=t-1$, and set without loss of generality that $X=\{(1,1),\ldots,(t,t)\}$ and $L=(1\ n)$. So
$$\mathcal{D}[X]\subseteq\mathcal{H}:=\{\sigma\in S_n:\sigma(i)=i\;\mbox{for all}\;i\in[t],\;\mbox{and}\;
\sigma(i)=i\text{ for some }i\in[t+1,n-1]\}.$$
 Let us count the number of $F\in S_n$ with $X\subseteq F$ and $F\cap(L\setminus X)=\emptyset$. Fix such an $F$. Then $F$ fixes $[t]$, and $F(i)\neq i$ for any $i\in[t+1,n-1]$. If $F(n)=n$, then $F$ acts as a derangement on $[t+1,n-1]$; if $F(n)\neq n$, then there are $d_{n-t}$ choices for the remaining points. Hence
\begin{equation*}
|\{F\in S_n:X\subseteq F,\;F\cap(L\setminus X)=\emptyset\}|=d_{n-t-1}+d_{n-t}.
\end{equation*}
So $|\mathcal{D}[X]|\leq|\mathcal{H}|=(n-t)!-d_{n-t}-d_{n-t-1}=h(n,t)$, and it is easy to characterize equality.

Finally, let us prove (iii). By our assumption above, $X^{(1)}=[t]$ and  $k=n$, and $K$ fixes all but one point of $[t]$. Moreover, there is a $j\in[t+1,n-1]$ with $K(j)\neq j$. To bound $|\mathcal D[X]|$ from above, we estimate the number of permutations in $\mathcal{H}$ which disagree with $K$ at every point greater than $t$. The family
$$\{\sigma\in S_n:\sigma(i)=i\;\mbox{for all}\;i\in[t]\cup\{j\},\;\sigma(i)\neq K(i)\;\mbox{for all}\;i\in[t+1,n]\setminus\{j\}\}$$
forms a collection of such permutations. After fixing
$[t]\cup\{j\}$, at most $n-t-1$ assignments
of $K$ remain feasible. By Lemma \ref{lemmadrgmt} (i), there are at least
$d_{n-t-1}$ such permutations. Hence
$$|\mathcal D[X]|\leq|\mathcal{H}|-d_{n-t-1}=h(n,t)-d_{n-t-1},$$
which proves (iii). 
\end{proof}
We next give a key estimate in the $t$-cover method.
It bounds the maximum $t$-degree of a family using a
collection of small $t$-covers with large $t$-covering number. 
Similar arguments proved to be useful for other objects (see, e.g., \cite{Cao-Lv-Wang-2021,Cao-Lv-Wang-Zhou-2022,Wen-Lv-2026}).
\begin{lemma}\label{lemmatcoverkey-perm}
	Let $\mathcal{D}\subseteq S_n$, and let $\mathcal{C}$ be a collection of partial permutations, each of size at most $\ell$, such that every member of $\mathcal{C}$ is a $t$-cover of $\mathcal{D}$. If $\tau_t(\mathcal{C})\geq m$ and $n\geq m+\ell+1$, then $$\Delta_t(\mathcal{D})\leq(\ell-t+1)^{m-t}(n-m)!.$$
\end{lemma}
\begin{proof}
	The assertion is immediate when $m=t$. Suppose that $m\geq t+1$ and fix a partial permutation $X$ of size $t$ with $|\mathcal{D}[X]|=\Delta_t(\mathcal{D})$. Put $X_0=X$. As long as $|X_i|<m$, the set $X_i$ is not a $t$-cover of $\mathcal{C}$. Hence, there exists $C_i\in\mathcal{C}$ such that $s_i:=|X_i\cap C_i|<t$. Set 
	$$\mathcal{H}_i:= \left\{ H\in\binom{C_i\setminus X_i}{t-s_i}: X_i\cup H\;\mbox{ is a partial permutation} \right\}.$$
Every member of $\mathcal{D}[X_i]$ contains at least $t-s_i$ elements of $C_i\setminus X_i$. So for some $H_i\in\mathcal{H}_i$,
	$$
	|\mathcal{D}[X_i]|
	\leq
	\binom{|C_i|-s_i}{t-s_i}
	|\mathcal{D}[X_i\cup H_i]|
	\leq
	(\ell-t+1)^{t-s_i}
	|\mathcal{D}[X_i\cup H_i]|.
	$$
	
	Set $X_{i+1}=X_i\cup H_i$ and continue the procedure until the first index $a$ for which $|X_a|\geq m$. Since $t-s_i=|X_{i+1}|-|X_i|$ for $i\leq a-1$, we obtain
	$$
	|\mathcal{D}[X]|
	\leq
	(\ell-t+1)^{|X_a|-t}|\mathcal{D}[X_a]|
	\leq
	(\ell-t+1)^{|X_a|-t}(n-|X_a|)!.
	$$
	Moreover, $
	m\leq |X_a|\leq m+t-1$, as $|X_{a-1}|<m$ and $|X_a|=|X_{a-1}|+t-s_{a-1}$ by our choice of $a$. 
	For $m\leq z\leq m+t-1$, the function $
	(\ell-t+1)^{z-t}(n-z)!$ is decreasing, since $n\geq m+\ell+1$. Consequently, $\Delta_t(\mathcal{D})=|\mathcal{D}[X]|\leq(\ell-t+1)^{m-t}(n-m)!$.
\end{proof}
\subsection{Estimates for peeling layers}
We now fix the notation used throughout the remainder of this section.
\begin{assumption}\label{assumption}
	Let $t\geq1$, $n\geq Ct$ and $q=t+\left\lceil3\ln\binom{n}{t}\right\rceil$, where $C:=400$. In particular, $n\geq16q$. Suppose that $(\mathcal F,\mathcal G)$ is a non-trivial pair of cross $t$-intersecting families in $S_n$. Let $(\mathcal{S},\mathcal{T})$ be the pair given by applying Theorem \ref{thmcrossspreadapp} to $(\mathcal{F},\mathcal{G})$. Further, if  $\mathcal{S},\mathcal{T}\neq\emptyset$, apply Algorithm \ref{algo} to $\Omega=[n]^2$, the pair $(\mathcal{S},\mathcal{T})$, the uniformity $q$ and an arbitrary initial fingerprint $(\mathcal{S}_0,\mathcal{T}_0)$. Let $N$ be the number of rounds and $(\mathcal{S}_{i},\mathcal{T}_{i},\mathcal{X}_{i},\mathcal{Y}_i)$ $(i\leq N)$ be the output families.
\end{assumption}
For convenience, we set $\mathcal{X}_i=\mathcal{Y}_i=\emptyset$ for $i>N$, and display the bound given by Theorem \ref{thmcrossspreadapp} (i) as follows.
\begin{equation}\label{equremainder}
\max\{|\mathcal{F}\setminus\mathcal{F}[\mathcal{S}]|,|\mathcal{G}\setminus\mathcal{G}[\mathcal{T}]|\}\leq\binom{n}{t}^{-2}(n-t)!=:R.
\end{equation}
\begin{lemma}\label{lemmaempty}
If $\mathcal{S}=\emptyset$ or $\mathcal{T}=\emptyset$, then $\min\{|\mathcal{F}|,|\mathcal{G}|\}<h(n,t)$ and $|\mathcal{F}||\mathcal{G}|<h(n,t)^2$.
\end{lemma}
\begin{proof}
Suppose without loss of generality that $\mathcal{S}=\emptyset$. Then by \eqref{equremainder} and Lemma \ref{lemmadecomp} (ii), $$|\mathcal{F}|=|\mathcal{F}\setminus\mathcal{F}[\mathcal{S}]|\leq R=\binom{n}{t}^{-2}(n-t)!<h(n,t).$$
For the product of sizes, we simply bound  $|\mathcal{G}|$ from above by 
$\binom{n}{t}(n-t)!$. Hence
$$|\mathcal{F}||\mathcal{G}|/((n-t)!)^2<\binom{n}{t}^{-1}<0.36,$$
and so $|\mathcal{F}||\mathcal{G}|<h(n,t)^2$.
\end{proof}
By Lemma \ref{lemmafingerprintproperty} (ii), for each $m\leq N$, we have
\begin{align}
\mathcal{F}&=\left(\bigcup_{i=0}^{m}\mathcal{F}[\mathcal{X}_i]\right)\cup\mathcal{F}[\mathcal{S}_m\setminus\mathcal{X}_m]\cup(\mathcal{F}\setminus\mathcal{F}[\mathcal{S}]).\label{equdecompf}\\
\mathcal{G}&=\left(\bigcup_{i=0}^{m}\mathcal{G}[\mathcal{Y}_i]\right)\cup\mathcal{G}[\mathcal{T}_m\setminus\mathcal{Y}_m]\cup(\mathcal{G}\setminus\mathcal{G}[\mathcal{T}]).\label{equdecompg}
\end{align}
\begin{lemma}\label{lemmafin-stru0}
We have  $$\max\left\{\left|\bigcup_{i=0}^{m}\mathcal{\mathcal{F}}[\mathcal{X}_i]\right|, \left|\bigcup_{i=0}^{m}\mathcal{\mathcal{G}}[\mathcal{Y}_i]\right|\right\}\leq U_m\cdot(n-t-1)!\;\;\;\mbox{for}\;\;0\leq m\leq q-t-1,$$
where $U_m:=0.1t$ for $m\leq q-t-2$ and $U_{q-t-1}:=5t$. In particular, if $N<q-t$, then $\min\{|\mathcal{F}|,|\mathcal{G}|\}<5t(n-t-1)!<h(n,t)$.
\end{lemma}
\begin{proof}
By Lemma \ref{lemmafingerprintproperty} (i) and (v), we have
\begin{align}
\left|\bigcup_{i=0}^{m}\mathcal{\mathcal{F}}[\mathcal{X}_i]\right|&\leq\sum_{i=0}^{m}\sum_{W\in\mathcal{X}_i}|S_n[W]|\leq\sum_{i=0}^{m}|\mathcal{X}_i|(n-(q-i))!\nonumber\\
&\leq\sum_{j=q-t-m}^{q-t}\binom{j+t}{t}(j+1)^{j}(n-t-j)!.\nonumber
\end{align}
Write $a_j:=\binom{j+t}{t}(j+1)^{j}(n-t-j)!$ for short. Then for $j\leq q-t-1$,
\begin{align*}
\frac{a_{j+1}}{a_{j}}=\frac{j+t+1}{n-t-j}\cdot\left(1+\frac{1}{j+1}\right)^{j+1}<\frac{eq}{n-q+1}=:\varepsilon.
\end{align*}
Then
\begin{equation}\label{equlemmafin-stru01}
\left|\bigcup_{i=0}^{m}\mathcal{\mathcal{F}}[\mathcal{X}_i]\right|\leq\sum_{j=q-t-m}^{q-t}a_j\leq a_{q-t-m}/(1-\varepsilon)\leq a_2/(1-\varepsilon)
\end{equation}
for $m\leq q-t-2$, and $\sum_{j=1}^{q-t}a_j\leq a_2/(1-\varepsilon)+a_1$.
Since $n\ge16q$, we have $
\varepsilon=\frac{eq}{n-q+1}<\frac e{15}<\frac12$. Moreover, $n-t-1\ge398t$ and $(t+1)(t+2)\le6t^2$, so
\[
\frac{a_2}{(1-\varepsilon)(n-t-1)!}
=\frac{9\binom{t+2}{2}}{(n-t-1)(1-\varepsilon)}
\le\frac{27t}{398(1-e/15)}
<0.1t.
\]
Hence $
\frac{a_2}{1-\varepsilon}+a_1
<(2.1t+2)(n-t-1)!$. This together with (\ref{equlemmafin-stru01}) yields  $\left|\bigcup_{i=0}^{m}\mathcal{\mathcal{F}}[\mathcal{X}_i]\right|\leq U_m(n-t-1)!$. By the same argument, the estimate holds also for $\left|\bigcup_{i=0}^{m}\mathcal{\mathcal{G}}[\mathcal{Y}_i]\right|$. 

Suppose $N<q-t$, and assume without loss of generality that $\mathcal{S}_N=\mathcal{X}_N$. Then by combining (\ref{equdecompf}), (\ref{equremainder}) and the bound above, we obtain that
\begin{equation*}
|\mathcal{F}|\leq\left|\bigcup_{i=0}^{N}\mathcal{\mathcal{F}}[\mathcal{X}_i]\right|+|\mathcal{F}\setminus\mathcal{F}[\mathcal{S}]|\leq(2.1t+2)(n-t-1)!+\binom{n}{t}^{-2}(n-t)!<5t(n-t-1)!,
\end{equation*}
as desired.
\end{proof}
To estimate the product, we follow decomposition is crucial.
\begin{lemma}\label{lemmadecomp}
Suppose $\mathcal{S}_N=\mathcal{X}_N$. Then
	$$\mathcal F[\mathcal S]\times\mathcal G[\mathcal T]\subseteq\bigcup_{i=0}^{N}\bigl(\mathcal F[\mathcal X_i]\times\mathcal G[\mathcal T_i]\bigr)\cup\bigcup_{i=0}^{N-1}\bigl(\mathcal F[\mathcal S_i]\times\mathcal G[\mathcal Y_i]\bigr).$$
\end{lemma}
\begin{proof}
If $N=0$, then $\mathcal{S}_0=\mathcal{X}_0$, and then the containment property of a fingerprint (Lemma \ref{lemmafingerprintdef} (ii)) gives 
\begin{equation*}
	\mathcal{F}[\mathcal{S}]\times\mathcal{G}[\mathcal{T}]
	\subseteq
	\mathcal{F}[\mathcal{S}_0]\times\mathcal{G}[\mathcal{T}_0]
	=
	\mathcal{F}[\mathcal{X}_0]\times\mathcal{G}[\mathcal{T}_0].
\end{equation*}
So the lemma holds trivially. Hence, we may assume $N\geq1$.
	
Let $(F,G)\in\mathcal{F}[\mathcal{S}]\times\mathcal{G}[\mathcal{T}]$, and choose 
	$S\in\mathcal{S}$ and $T\in\mathcal{T}$ with $S\subseteq F$ and $T\subseteq G$. By Lemma \ref{lemmafingerprintdef} (ii) again, there are $S_0\in\mathcal{S}_0$ and $T_0\in\mathcal{T}_0$ with $S_0\subseteq S$ and $T_0\subseteq T$. For each $i<N$ and $S_i\in\mathcal S_i\setminus\mathcal X_i$, pick $S_{i+1}\in\mathcal{S}_{i+1}$ with $S_{i+1}\subseteq S_i$. Since $\mathcal{S}_N=\mathcal{X}_N$, the procedure ends with some index $u\leq N$ and a chain of subsets 
	$$S\supseteq S_0\supseteq S_1\supseteq\cdots\supseteq S_{u},$$
	where $S_i\in\mathcal{S}_i\setminus\mathcal{X}_i$ for $i<u$ and $S_{u}\in\mathcal{X}_{u}$. This yields $F\in\mathcal{F}[\mathcal{X}_u]$. Next, we define an index and a chain $T\supseteq T_0\supseteq T_1\supseteq\cdots\supseteq T_{v}$ as follows. If the same procedure as above ends with some $T_{i}\in\mathcal{Y}_{i}$ and $i\leq N-1$, then set $v=i$; if we have got a chain $T\supseteq T_0\supseteq T_1\supseteq\cdots\supseteq T_{N-1}$ with $T_i\in\mathcal{T}_i\setminus\mathcal{Y}_i$ for $i\leq N-1$, then just put $v=N$ and pick a $T_N\in\mathcal{T}_N$ with $T_{N-1}\supseteq T_N$. Now if $u\leq v$, then $(F,G)\in\mathcal{F}[\mathcal{X}_u]\times\mathcal{G}[\mathcal{T}_u]$ as $G\supseteq T_u$. If $u>v$, then $v\leq N-1$ and hence $T_v\in\mathcal{Y}_v$, implying that $(F,G)\in\mathcal{F}[\mathcal{S}_v]\times\mathcal{G}[\mathcal{Y}_v]$.
\end{proof}
Next, we combine this decomposition with Lemma \ref{lemmatcoverkey-perm}.
At each level, the $t$-covering number of $\mathcal{X}_i$ or $\mathcal{Y}_i$ controls both the size of one layer $\mathcal{F}[\mathcal{X}_i]$ or $\mathcal{G}[\mathcal{Y}_i]$ and the $t$-degree of the opposite
family $\mathcal{G}[\mathcal{T}_i]$ or $\mathcal{F}[\mathcal{S}_i]$.
\begin{lemma}\label{lemmagroupproduct}
If $N<q-t$, then $|\mathcal{F}||\mathcal{G}|<h(n,t)^2$.
\end{lemma}
\begin{proof}
Without loss of generality, suppose $\mathcal{S}_N=\mathcal{X}_N$.	We estimate each pair in the right-hand side of the decomposition in Lemma \ref{lemmadecomp}. Fix $i\leq N$, and consider the pair  $(\mathcal{F}[\mathcal{X}_i],\mathcal{G}[\mathcal{T}_i])$ of cross $t$-intersecting families. Write $j=q-t-i$ for short. We may suppose that $\mathcal X_i\neq\emptyset$ and $\mathcal G[\mathcal T_i]\neq\emptyset$. Since every member of $\mathcal{T}_i$ is a $t$-cover of
$\mathcal X_i$ and has size at most $q-i=t+j$, it follows that $\tau_t(\mathcal X_i)\leq t+j$. 
Let $C$ be a $t$-cover of $\mathcal X_i$ of size
$\tau_t(\mathcal X_i)$. Every member of $\mathcal X_i$
contains a $t$-subset of $C$. Hence, by
Lemma \ref{lemmafingerprintproperty} {\rm(iv)},
$$
|\mathcal X_i|
\leq
\sum_{Z\in\binom Ct}|\mathcal X_i[Z]|
\leq
\binom{\tau_t(\mathcal X_i)}t(j+1)^j.
$$
It follows that
\begin{equation}\label{equproductXi}
	|\mathcal F[\mathcal X_i]|
	\leq
	\binom{\tau_t(\mathcal X_i)}t
	(j+1)^j(n-t-j)!.
\end{equation}
On the other hand, $\mathcal{X}_i$ forms a collection of $t$-covers of $\mathcal{G}[\mathcal{T}_i]$. By applying
Lemma \ref{lemmatcoverkey-perm} with $\mathcal{D}=\mathcal{G}[\mathcal{T}_i]$, $\mathcal{C}=\mathcal{X}_i$, $\ell=t+j$ and $m=\tau_t(\mathcal{X}_i)$, we obtain
$$
\Delta_t(\mathcal{G}[\mathcal{T}_i])\leq(j+1)^{\tau_t(\mathcal X_i)-t}
\bigl(n-\tau_t(\mathcal X_i)\bigr)!.$$
Fixing a member of \(\mathcal X_i\) and taking the union over its
\(t\)-subsets gives
\begin{equation}\label{equproductYi}
	|\mathcal G[\mathcal T_i]|
	\leq
	\binom{t+j}{t}
	(j+1)^{\tau_t(\mathcal X_i)-t}
	\bigl(n-\tau_t(\mathcal X_i)\bigr)!.
\end{equation}
Since $t\le\tau_t(\mathcal X_i)\le t+j\le q$, we have
\[
\frac{(n-t-j)!}{(n-t)!}\le(n-q+1)^{-j}\;\;\mbox{and}\;\;
\frac{(n-\tau_t(\mathcal X_i))!}{(n-t)!}
\le(n-q+1)^{-(\tau_t(\mathcal X_i)-t)}.
\]
By combining these with \eqref{equproductXi} and \eqref{equproductYi}, we obtain that
\begin{equation}\label{equproductratio}
\frac{|\mathcal F[\mathcal X_i]|
		|\mathcal G[\mathcal T_i]|}
	{((n-t)!)^2}\leq\frac{1}{j!\,(\tau_t(\mathcal X_i)-t)!}\left(\frac{q(j+1)}{n-q+1}\right)^{j+\tau_t(\mathcal X_i)-t}.
\end{equation}
 Since $n\ge16q$, we have $q/(n-q+1)<1/15$.
Recall that $t\le\tau_t(\mathcal X_i)\le t+j$.
If $j=1$, then $\tau_t(\mathcal X_i)-t\in\{0,1\}$,
so the right-hand side in (\ref{equproductratio}) is at most $2/15$. For $j\geq2$, we simply use
\begin{equation*}
	\frac{1}{(\tau_t(\mathcal{X}_i)-t)!}
	\left(\frac{j+1}{15}\right)^{\tau_t(\mathcal{X}_i)-t}
	\leq
	\sum_{s=0}^{\infty}\frac{1}{s!}
	\left(\frac{j+1}{15}\right)^s
	=
	e^{(j+1)/15}.
\end{equation*}
Combining this with (\ref{equproductratio}) and $q/(n-q+1)<1/15$, we obtain
\begin{equation}\label{equproductratio'}
\frac{|\mathcal{F}[\mathcal{X}_i]||\mathcal{G}[\mathcal{T}_i]|}{((n-t)!)^2}
\leq\frac{1}{j!}\left(\frac{j+1}{15}\right)^je^{(j+1)/15}.
\end{equation}
These bounds decrease geometrically, as the ratio for successive
values of $j$ is
\[
\frac{e^{1/15}}{15}
\left(1+\frac1{j+1}\right)^{j+1}
<\frac{e^{16/15}}{15}<\frac15.
\]
Interchanging the two families gives the same bounds for
$|\mathcal F[\mathcal S_i]|\,|\mathcal G[\mathcal Y_i]|/((n-t)!)^2$.

By Lemma \ref{lemmadecomp}, we have 
$$|\mathcal F[\mathcal S]||\mathcal G[\mathcal T]|\leq\sum_{i=0}^{N}|\mathcal F[\mathcal X_i]||\mathcal G[\mathcal T_i]|+\sum_{i=0}^{N-1}|\mathcal F[\mathcal S_i]||\mathcal G[\mathcal Y_i]|.$$
Note that $N\leq q-t-1$ and $j=q-t-i$. Then the term with $j=1$ occurs
at most once, and each term with $j\ge2$ occurs at most twice.
Since the right-hand side in (\ref{equproductratio'}) at $j=2$ is $e^{1/5}/50$, we obtain
\[
\begin{aligned}
	\frac{|\mathcal F[\mathcal S]|\,
		|\mathcal G[\mathcal T]|}{((n-t)!)^2}
	&\le\frac2{15}
	+2\cdot\frac{e^{1/5}}{50}
	\sum_{j=2}^{\infty}\left(\frac15\right)^{j-2}=\frac2{15}+\frac{e^{1/5}}{20}<0.2.
\end{aligned}
\]

Finally, this together with \eqref{equremainder} yields
\begin{align*}
|\mathcal{F}||\mathcal{G}|&\leq|\mathcal{F}[\mathcal{S}]||\mathcal{G}[\mathcal{T}]|+R(|\mathcal{F}[\mathcal{S}]|+|\mathcal{G}[\mathcal{T}]|)+R^2\\
&<0.2((n-t)!)^2+2R\binom{n}{t}(n-t)!+R^2<0.3((n-t)!)^2,
\end{align*}
where in the last step we used $\frac{2R\binom{n}{t}(n-t)!}{((n-t)!)^2}=2\binom{n}{t}^{-1}<0.05$ and $R^2/((n-t)!)^2=\binom{n}{t}^{-4}<0.05$ for $n\geq400t$. Then $|\mathcal{F}||\mathcal{G}|<h(n,t)^2$ as Lemma \ref{lemmadrgmt} (ii) gives $h(n,t)\geq0.6(n-t)!$.
\end{proof}
\subsection{The terminal case $N=q-t$ and the product bounds}
Lemmas \ref{lemmafin-stru0} and \ref{lemmagroupproduct} give strict bounds for  $N\leq q-t-1$. In the remaining case $N=q-t$, the peeling procedure reaches the $t$-uniform layer. 
\begin{lemma}\label{lemmafin-stru2}
	Suppose $N=q-t$. Then there exist a partial permutation $X$ of size $t$ and an index $j$ with $j=\min\{0\leq i\leq q-t-1:\mathcal{S}_{i}=\mathcal{T}_{i}=\{X\}\}$, and the following hold.
\begin{itemize}
\item[\rm(i)]$\min\{|\mathcal{F}|,|\mathcal{G}|\}\leq h(n,t)+t$, with equality if and only if $(\mathcal F,\mathcal G)$ is isomorphic to a pair given in
Construction \ref{construction}. Further, if both $\mathcal{F}$ and $\mathcal{G}$ are non-trivial, then $|\mathcal{F}||\mathcal{G}|\leq(h(n,t)+t)^2$, with the same extremal configurations.
\item[\rm (ii)] $|\mathcal{F}||\mathcal{G}|\leq h(n,t)((n-t)!+t)$, with equality if and only if $(\mathcal F,\mathcal G)$ or $(\mathcal G,\mathcal F)$ is isomorphic to a pair given in
Construction \ref{construction2}. 
\end{itemize}	
\end{lemma}
\begin{proof}
We will use frequently the following estimate given in Lemma \ref{lemmadrgmt} (ii).
\begin{equation}\label{equfunhbound}
0.6(n-t)!\leq h(n,t)\leq0.66(n-t)!.
\end{equation}
This also yields $(h(n,t)+t)^2<h(n,t)((n-t)!+t)$.

Since $N=q-t$, we have $\mathcal{X}_{q-t-1}\subsetneqq\mathcal{S}_{q-t-1}$ and $\mathcal{Y}_{q-t-1}\subsetneqq\mathcal{T}_{q-t-1}$. Then both $\mathcal{S}_{q-t-1}$ and $\mathcal{T}_{q-t-1}$ contain at least one $t$-subset, and then $\mathcal{S}_{q-t-1}=\mathcal{T}_{q-t-1}=\{X\}$ because these two families are cross $t$-intersecting antichains.  So we may consider the minimal index $j$ with this property.

\noindent{\bf Case 1.\;}$j>0$.

By the minimality of $j$, the families $\mathcal S_{j-1}$ and
 $\mathcal T_{j-1}$ are not both equal to $\{X\}$. Since they are
 cross $t$-intersecting antichains, they cannot both contain $X$. Then we may assume without loss of generality that $X\notin\mathcal{S}_{j-1}$. Write $\mathcal{S}^\uparrow=\mathcal{S}_{j-1}\setminus\mathcal{X}_{j-1}$ and $\mathcal{T}^\uparrow=\mathcal{T}_{j-1}\setminus\mathcal{Y}_{j-1}$ for short. Then $X$ is contained in every member of $\mathcal{S}^\uparrow$, as $(\mathcal{S}_{j},\mathcal{T}_{j})=(\{X\},\{X\})$ is a fingerprint of  $(\mathcal{S}^\uparrow,\mathcal{T}^\uparrow)$. Therefore, from (\ref{equdecompf}), we have
\begin{equation}\label{equlemmafin-stru21}
\mathcal{F}\subseteq\left(\bigcup_{i=0}^{j-1}\mathcal{F}[\mathcal{X}_i]\right)\cup\mathcal{F}[\mathcal{S}^\uparrow]\cup(\mathcal{F}\setminus\mathcal{F}[\mathcal{S}]).
\end{equation}

 Since $\mathcal{S}_{j-1}$ is an antichain and $X$ is contained in every $S\in\mathcal{S}^\uparrow$, it follows that $X$ is a proper subset of each set from $\mathcal{S}^\uparrow$. So by Lemma \ref{lemmafingerprintdef} (iii), $X$ is not a $t$-cover of $\mathcal{T}_{j-1}$. Hence there exists $W_0\in\mathcal{T}_{j-1}$ with $|W_0\cap X|\leq t-1$. Of course $W_0\in\mathcal{Y}_{j-1}$, because every member of $\mathcal{T}^\uparrow$ also contains $X$. On the other hand, each  $F\in\mathcal{F}[\mathcal{S}^\uparrow]$ satisfies $X\subseteq F$ and $|F\cap W_0|\geq t$. By applying Lemma \ref{lemmaind} (i) with $\ell=q-j+1\leq q$, we obtain that $|\mathcal{F}[\mathcal{S}^\uparrow]|\leq q(n-t-1)!$. This together with (\ref{equlemmafin-stru21}), Lemma \ref{lemmafin-stru0} and (\ref{equremainder}) yields
\begin{align*}
|\mathcal{F}|&\leq0.1t(n-t-1)!+q(n-t-1)!+\binom{n}{t}^{-2}(n-t)!<(q+0.2t)(n-t-1)!.
\end{align*}
Hence $|\mathcal{F}|<h(n,t)$ from (\ref{equfunhbound}) and  $n\geq16q$. For $|\mathcal{G}|$, we simply use \eqref{equdecompg}, \eqref{equremainder} and Lemma \ref{lemmafin-stru0} to derive
\begin{align*}
|\mathcal{G}|&\leq\left|\bigcup_{i=0}^{j-1}\mathcal{\mathcal{G}}[\mathcal{Y}_i]\right|+|\mathcal{G}[\mathcal{T}^\uparrow]|+|\mathcal{G}\setminus\mathcal{G}[\mathcal{T}]|\leq0.2t(n-t-1)!+(n-t)!,
\end{align*}
where in the second step we used $\mathcal{G}[\mathcal{T}^\uparrow]\subseteq S_n[X]$. Hence
$$|\mathcal{F}||\mathcal{G}|/((n-t)!)^2<1.2(q+0.2t)/(n-t-1)<0.36,$$
which gives $|\mathcal{F}||\mathcal{G}|<h(n,t)^2$. Therefore, if $j>0$, then the inequalities in (i) and (ii) hold strictly.

\noindent{\bf Case 2.\;}$j=0$.

Now  $\mathcal{S}_0=\mathcal{T}_0=\{X\}$, and hence the decompositions (\ref{equdecompf}) and (\ref{equdecompg})  reduce to $$\mathcal{F}=\mathcal{F}[X]\cup(\mathcal{F}\setminus\mathcal{F}[X])\;\;\;\mbox{and}\;\;\mathcal{G}=\mathcal{G}[X]\cup(\mathcal{G}\setminus\mathcal{G}[X]).$$
Moreover, $
|\mathcal F\setminus\mathcal F[X]|,\,
|\mathcal G\setminus\mathcal G[X]|
\leq R=\binom{n}{t}^{-2}(n-t)!$.

We may assume that every member of $\mathcal{F}\setminus\mathcal{F}[X]$ or $\mathcal{G}\setminus\mathcal{G}[X]$ contains all but exactly one element of $X$. Indeed, if, say, $|L\cap X|\leq t-2$ for some $L\in\mathcal{F}\setminus\mathcal{F}[X]$, then Lemma \ref{lemmaind} (ii) gives $|\mathcal{G}[X]|\leq0.5(n-t)!$, and hence $|\mathcal{G}|\leq0.5(n-t)!+R<h(n,t)$, which proves the first inequality in (i). For the product of sizes, we first have a trivial bound $|\mathcal{F}|\leq(n-t)!+R$. So $|\mathcal{F}||\mathcal{G}|\leq((n-t)!+R)(0.5(n-t)!+R)<h(n,t)((n-t)!+t)$. If further both $\mathcal{F}$ and $\mathcal{G}$ are non-trivial, then we use Lemma \ref{lemmaind} (ii) to bound $|\mathcal{F}[X]|$, and then obtain
$$|\mathcal{F}||\mathcal{G}|\leq(h(n,t)+R)(0.5(n-t)!+R)<(h(n,t)+t)^2.$$
Thus both (i) and (ii) hold strictly in this case. Hence, in what follows, suppose 
$$|H\cap X|\geq t-1\;\;\mbox{for all}\;\;H\in\mathcal{F}\cup\mathcal{G}.$$

\noindent{\bf Case 2.1.\;}At least one of $\mathcal{F}$ and $\mathcal{G}$ is trivial.

Without loss of generality, suppose that there is a partial permutation $X'$ of size $t$ contained in every member of $\mathcal F$. Suppose first that $X'\ne X$. 
If $X\cup X'$ is not a partial permutation, then
$\mathcal F[X]=\emptyset$; otherwise, $X\cup X'$ is a partial
permutation of size at least $t+1$, and hence $
|\mathcal F[X]|\le (n-t-1)!$. Then $|\mathcal F|\le (n-t-1)!+R$ and $|\mathcal{F}||\mathcal{G}|\leq((n-t-1)!+R)((n-t)!+R)<h(n,t)^2$. Thus both required inequalities are strict in this case. So we consider the case that $X'=X$. Then $X\subseteq F$ for all $F\in\mathcal{F}$. Without loss of generality, assume that $X=\{(1,1),\ldots,(t,t)\}$, namely, every member of $\mathcal{F}$ fixes $[t]$. Since $(\mathcal{F},\mathcal{G})$ is non-trivial, we have  $\mathcal{G}\setminus\mathcal{G}[X]\neq\emptyset$. Fix a $T_0\in\mathcal{G}\setminus\mathcal{G}[X]$. By our assumption that $|T_0\cap X|=t-1$, we may set by double translate that $T_0=(1\ n)$. Then by Lemma \ref{lemmaind} (ii), $$|\mathcal{F}|=|\mathcal F[X]|\leq h(n,t)<h(n,t)+t,$$
which verifies the first inequality in (i). We proceed by bounding the product of their sizes. If there exists $T\in\mathcal{G}\setminus\mathcal{G}[X]$ such that $T(j)\neq j$ for some $j\in[t+1,n-1]$, then by applying Lemma \ref{lemmaind} (iii) with $\mathcal D=\mathcal F,L=(1\ n),K=T$ and $k=n$, we get $|\mathcal{F}|=|\mathcal F[X]|\leq h(n,t)-d_{n-t-1}$. Hence $|\mathcal{F}||\mathcal{G}|\leq(h(n,t)-d_{n-t-1})((n-t)!+R)$. It follows that
\begin{align*}
|\mathcal{F}||\mathcal{G}|-h(n,t)((n-t)!+t)&\leq h(n,t)(R-t)-d_{n-t-1}((n-t)!+R)\\
&<h(n,t)R-d_{n-t-1}(n-t)!<0.
\end{align*}
To prove the inequality above, we used Lemma \ref{lemmadrgmt} (ii), which gives  $d_{n-t-1}>0.3(n-t-1)!>R$.

It remains to suppose that $T(j)=j$ for all $j\in[t+1,n-1]$ and all $T\in\mathcal{G}\setminus\mathcal{G}[X]$. Recall again that we assumed $|T\cap X|=t-1$ for all such $T$. It follows that 
$$\mathcal{G}\setminus\mathcal{G}[X]\subseteq\{(i\ n):i\in[t]\}.$$
On the other hand, every member of $\mathcal{F}$ fixes $[t]$ and $t$-intersects with $T_0=(1\ n)$. Hence
\begin{align*}
	\mathcal{F}&\subseteq\{\sigma\in S_n:\sigma(i)=i\;\mbox{for all}\;i\in[t],\;\sigma(j)=j\;\mbox{for some}\;j\in[t+1,n-1]\},\;\;\mbox{and}\\
	\mathcal{G}&\subseteq\{\sigma\in S_n:\sigma(i)=i\;\mbox{for all}\;i\in[t]\}\cup\{(i\ n):i\in[t]\}.
\end{align*}
Thus $|\mathcal{F}||\mathcal{G}|\leq h(n,t)((n-t)!+t)$, and equality holds precisely if $\mathcal{F}$ and $\mathcal{G}$ equal the families above, respectively.

\noindent{\bf Case 2.2.\;}Both  $\mathcal{F}$ and $\mathcal{G}$ are non-trivial. 

Now  $\mathcal{F}\setminus\mathcal{F}[X]$ and $\mathcal{G}\setminus\mathcal{G}[X]\neq\emptyset$. Fix an $L\in\mathcal{F}\setminus\mathcal{F}[X]$. By double translate, we may suppose $X=\{(1,1),\ldots,(t,t)\}$ and  $L=(1\ n)$. For simplicity, denote by
$$\mathcal{C}:=S_n[X]=\{\sigma\in S_n:\sigma(i)=i\;\mbox{for all}\;i\in[t]\}$$
the $t$-coset centered at $X$. Every member of $\mathcal G[X]$ already agrees with $L=(1\ n)$
at the points $2,\ldots,t$. It cannot agree with $L$ at $1$ or
$n$, and hence must fix some point of $[t+1,n-1]$. Consequently,
\begin{equation}\label{equlemmafin-stru22}
\mathcal{G}[X]\subseteq\{\sigma\in\mathcal{C}:
\sigma(i)=i\text{ for some }i\in[t+1,n-1]\}=:\mathcal{H}.
\end{equation}
Note that the same counting argument as in the proof of Lemma \ref{lemmaind} (ii) gives $$|\mathcal{H}|=(n-t)!-d_{n-t}-d_{n-t-1}=h(n,t).$$

Next, choose
$M\in\mathcal G\setminus\mathcal G[X]$. Let $i_0$ be the unique
point of $[t]$ not fixed by $M$, and suppose $M(k_0)=i_0$. Then 
\begin{equation}\label{equlemmafin-stru23}
	\mathcal{F}[X]\subseteq\{\sigma\in\mathcal{C}:
	\sigma(i)=M(i)\text{ for some }i\in[t+1,n]\setminus\{k_0\}\}.
\end{equation}
By setting $\tau:=(1\ i_0)M(1\ i_0)$, the family on the right-hand side can be written as 
$$\mathcal{H}_{\tau}:=\{\sigma\in\mathcal{C}:
\sigma(i)=\tau(i)\text{ for some }i\in[t+1,n]\}.$$
This is because $\tau(k_0)=1\neq\sigma(k_0)$ for all $\sigma\in\mathcal{C}$. In addition, note that $\tau(1)\ne1$ and $\tau$ fixes $2,\ldots,t$. So $|\tau\cap X|=t-1$, and we get again that $|\mathcal{H}_{\tau}|=h(n,t)$.

We proceed by proving a claim.
\begin{claim}\label{claimlemmafin-stru21}
	The following hold.
\begin{itemize}
\item[\rm(i)]Either $|\mathcal G[X]|\leq h(n,t)-d_{n-t-1}$ or 
$\mathcal F\setminus\mathcal F[X]
\subseteq\{(i\ n):i\in[t]\}$.
\item[\rm(ii)]Either $|\mathcal F[X]|\leq h(n,t)-d_{n-t-1}$ or $\mathcal G\setminus\mathcal G[X]
\subseteq\{(1\ i)\tau(1\ i):i\in[t]\}$.
\end{itemize}
\end{claim}
\begin{proof}
(i)\;Suppose that
$F_0\in\mathcal F\setminus\mathcal F[X]$ is not one of the
transpositions $(i\ n)$, $i\in[t]$. Note that $F_0$ fixes all but one
point of $[t]$ and $t$-intersects $L$. Then there must be some $j\in[t+1,n-1]$ with $F_0(j)\neq j$. By  applying Lemma \ref{lemmaind}{\rm(iii)} with $\mathcal D=\mathcal G,L=(1\ n),K=F_0$ and $k=n$, we obtain $|\mathcal G[X]|\leq|\mathcal{H}|-d_{n-t-1}=h(n,t)-d_{n-t-1}$. Hence (i) holds.

(ii)\;The proof is analogous, with $k_0=\tau^{-1}(1)$ playing the role of $n$ in (i). We observe that $(1\ i)\tau(1\ i), i=1,\ldots,t$ are exactly those $\sigma\in S_n$  which fix
exactly $t-1$ points of $[t]$ and satisfy $
\sigma(j)=\tau(j)$ for each $j\in[t+1,n]\setminus\{k_0\}$. Indeed, each $(1\ i)\tau(1\ i)$ certainly has these
properties. Conversely, assume $\sigma$ satisfies them, and
let $i$ be the unique point of $[t]$ not fixed by $\sigma$. Then we get $\sigma(i)=\tau(1)$ and $\sigma(k_0)=i$, and thus $\sigma=(1\ i)\tau(1\ i)$. Suppose that
$G_0\in\mathcal G\setminus\mathcal G[X]$ is not one of these
permutations. Then $G_0(j)\neq\tau(j)=M(j)$ for some $j\in[t+1,n]\setminus\{k_0\}$. By applying Lemma \ref{lemmaind}{\rm(iii)} with $\mathcal D=\mathcal F,L=M,K=G_0$ and $k=k_0$, we get $|\mathcal F[X]|\leq h(n,t)-d_{n-t-1}$.
\end{proof}
Suppose first that at least one of $\mathcal F[X]$ or $\mathcal G[X]$ has size at most $h(n,t)-d_{n-t-1}$. Then $$\min\{|\mathcal{F}|,|\mathcal{G}|\}\leq h(n,t)-d_{n-t-1}+R<h(n,t)-R,$$
 because Lemma \ref{lemmadrgmt} (ii) gives $d_{n-t-1}>0.3(n-t-1)!>2R$. Next, we bound the size of the other one from above by $h(n,t)+R$, and then derive 
 \begin{align*}
 |\mathcal{F}||\mathcal{G}|<(h(n,t)+R)(h(n,t)-R)<h(n,t)^2.
 \end{align*}

It remains to assume, from Claim \ref{claimlemmafin-stru21}, that 
$$\mathcal F\setminus\mathcal F[X]
\subseteq\{(i\ n):i\in[t]\}\;\;\mbox{and}\;\;\mathcal G\setminus\mathcal G[X]
\subseteq\{(1\ i)\tau(1\ i):i\in[t]\}.$$
It follows that $\max\{|\mathcal{F}|,|\mathcal{G}|\}\leq h(n,t)+t$. Then certainly $|\mathcal{F}||\mathcal{G}|\leq(h(n,t)+t)^2$. Finally, let us characterize equality. Suppose equality holds. Then
\begin{align*}
\mathcal F=&\{\sigma\in\mathcal{C}:\sigma(j)=\tau(j)\text{ for some }j\in[t+1,n]\}\cup\{(i\ n):i\in[t]\},\;\mbox{and}\\
\mathcal G=&\{\sigma\in\mathcal{C}:\sigma(j)=j\text{ for some }j\in[t+1,n-1]\}\cup\{(1\ i)\tau(1\ i):i\in[t]\}.
\end{align*}
It remains only to characterize $\tau$ via the property of cross
$t$-intersection. If $t\geq2$, compare $(i\ n)$ with
$(1\ j)\tau(1\ j)$ for distinct $i,j\in[t]$. They agree
at exactly $t-2$ points of $[t]$, while their common points  
outside $[t]$ are precisely the fixed points of $\tau$ in
$[t+1,n-1]$. Hence $\tau$ must fix at least two points
in this set. If $t=1$, the condition is simply
that $\tau$ and $(1\ n)$ intersect. So by the notation given in Construction \ref{construction}, we get 
$$(\mathcal{F},\mathcal{G})=(\mathcal{H}((1\ n),\tau),\mathcal{H}(\tau,(1\ n))).$$
This finishes the proof.
\end{proof}

\noindent{\bf Proof of Theorem \ref{thmpurelycombin}}.\;We adapt the notation in Assumption \ref{assumption}. Recall that the function $h(n,t)=(n-t)!-d_{n-t}-d_{n-t-1}$ is defined in (\ref{equsizeh}), and with this notation, the upper bounds in (i) and (ii) are $h(n,t)((n-t)!+t)$ and $(h(n,t)+t)^2$, respectively.

By Lemma \ref{lemmaempty}, we may suppose that both $\mathcal{S}$ and $\mathcal{T}$ are non-empty. Then (i) follows directly from Lemmas \ref{lemmagroupproduct} and \ref{lemmafin-stru2} (ii), and (ii) follows from Lemmas \ref{lemmagroupproduct} and \ref{lemmafin-stru2} (i). {\hfill$\square$}\vspace{1em}

\subsection{Bounds for $t$-diversities}
For a partial permutation $Z$ of size $t+2$, write
\[
\mathcal F_1(Z)
:=\{\sigma\in S_n:|\sigma\cap Z|\geq t+1\}.
\]

\noindent{\bf Proof of Theorem \ref{thmcrossdiv}}.\;If $\mathcal{F}$ or $\mathcal{G}$ is trivial, then we are done as $\min\{\gamma_t(\mathcal{F}),\gamma_t(\mathcal{G})\}=0$. Thus we may suppose that
both families are non-trivial and use the notation in Assumption
\ref{assumption}. Recall from \eqref{equremainder} that
\begin{equation}\label{equthmcrossdiv1}
	\max\{|\mathcal{F}\setminus\mathcal{F}[\mathcal{S}]|,|\mathcal{G}\setminus\mathcal{G}[\mathcal{T}]|\}\leq R=\binom{n}{t}^{-2}(n-t)!<n^{-1}(n-t-1)!.
\end{equation}
Since $n\geq400t$, the bound above gives
\begin{equation}\label{equthmcrossdiv2}
0.77t(n-t-1)!+R<t\bigl((n-t-1)!-(n-t-2)!\bigr)=\gamma_t(\mathcal F_1(n,t)).
\end{equation}

If $\mathcal S=\emptyset$ or $\mathcal T=\emptyset$, then the two inequalities displayed above immediately give the required bound. 
Hence suppose that both are non-empty and perform Algorithm
\ref{algo}.

If $N\leq q-t-2$, assume without loss of generality that
$\mathcal S_N=\mathcal X_N$. By \eqref{equdecompf},
Lemma \ref{lemmafin-stru0} and \eqref{equthmcrossdiv1}, $\gamma_t(\mathcal F)\leq|\mathcal F|\leq0.1t(n-t-1)!+R<\gamma_t(\mathcal F_1(n,t))$.

Suppose that $N=q-t$. Then Lemma \ref{lemmafin-stru2} yields $
\mathcal S_{q-t-1}=\mathcal T_{q-t-1}=\{X\}$ for some partial permutation $X$ of size $t$. Hence,
$$
\gamma_t(\mathcal F)
\leq|\mathcal F\setminus\mathcal F[X]|
\leq
\left|\bigcup_{i=0}^{q-t-2}
\mathcal F[\mathcal X_i]\right|+R
\leq0.1t(n-t-1)!+R<\gamma_t(\mathcal F_1(n,t)).
$$

It remains to suppose $N=q-t-1$. Assume without loss of generality that
$\mathcal S_N=\mathcal X_N$. 
%If $\mathcal S_N$ or $\mathcal T_N$ is empty, the preceding argument applied to the corresponding original family gives a strict inequality.

If $\mathcal T_N\neq\mathcal Y_N$, choose
$X\in\mathcal T_N\setminus\mathcal Y_N$. Since
$\mathcal S_N$ and $\mathcal T_N$ are cross $t$-intersecting, we have 
$|X|=t$ and $X\subseteq S$ for every $S\in\mathcal S_N$. Then we arrive again at 
$$
\gamma_t(\mathcal F)
\leq|\mathcal F\setminus\mathcal F[X]|
\leq
\left|\bigcup_{i=0}^{N-1}
\mathcal F[\mathcal X_i]\right|+R
\leq0.1t(n-t-1)!+R<\gamma_t(\mathcal F_1(n,t)).
$$
We may consequently suppose that
$\mathcal T_N=\mathcal Y_N$. Thus both $\mathcal S_N$ and
$\mathcal T_N$ are $(t+1)$-uniform.

Suppose first that there is a set $Z$ of size $t+2$ such that
$$
\mathcal S_N,\mathcal T_N\subseteq\binom Z{t+1}.
$$
If $Z$ is not a partial permutation, at most two of its
$(t+1)$-subsets are partial permutations. Hence all members of
one of $\mathcal S_N$ and $\mathcal T_N$ contain a common
$t$-partial permutation, and the argument above applies.
We may therefore assume that $Z$ is a partial permutation.

If $\mathcal F,\mathcal G\subseteq\mathcal F_1(Z)$, the required
inequality follows from the monotonicity of $t$-diversity, that is, adding members to a family does not
decrease its $t$-diversity.
Otherwise, suppose by symmetry that
$L\in\mathcal F\setminus\mathcal F_1(Z)$, and choose
$X\in\binom Zt$ such that $L\cap Z\subseteq X$. If every member of $\mathcal{T}_N$ contains $X$, then
$\mathcal{G}[\mathcal{T}_N]\subseteq\mathcal{G}[X]$. Hence,
by \eqref{equdecompg}, Lemma \ref{lemmafin-stru0} and
\eqref{equremainder},
$$
\gamma_t(\mathcal{G})\leq
\left|\bigcup_{i=0}^{N-1}\mathcal{G}[\mathcal{Y}_i]\right|+R<\gamma_t(\mathcal{F}_1(n,t)).
$$
We may therefore choose
$T\in\mathcal{T}_N$ with $X\nsubseteq T$. Then
$|L\cap T|<t$. After fixing $T$, let $P$ consist of those assignments of $L\setminus T$ whose rows and columns are not used by $T$. Then $|P|\leq n-t-1$, and every member of
$\mathcal G[T]$ must intersect $P$. It follows from Lemma \ref{lemmadecomp} that
\[
|\mathcal G[T]|
\leq(n-t-1)!-d(n-t-1,|P|)\leq
(n-t-1)!-d_{n-t-1}\leq0.67(n-t-1)!.
\]
There are at most $t$ members of $\binom Z{t+1}$ which do not
contain $X$. Hence, by \eqref{equdecompg},
Lemma \ref{lemmafin-stru0} and \eqref{equthmcrossdiv1}, we obtain that
\begin{align*}
\gamma_t(\mathcal G)&\leq|\mathcal G\setminus\mathcal G[X]|\leq\left|\bigcup_{i=0}^{N-1}\mathcal G[\mathcal Y_i]\right|
	+|\mathcal G[\mathcal T_N]\setminus\mathcal G[X]|+R\\
	&\leq0.77t(n-t-1)!+R<\gamma_t(\mathcal F_1(n,t)).
\end{align*}

It remains to suppose that no such $Z$ exists. The proof reduces to characterizing the structure of $(t+1)$-uniform fingerprints.
We use the following classification from \cite{Wen-Lv-2026+}.
\begin{lemma}[\cite{Wen-Lv-2026+}]\label{lemmastrue-alike}
	Let $t\geq2$, $n\geq t+3$ and  $\mathcal{S},\mathcal{T}\subseteq\binom{[n]}{t+1}$. Suppose that $(\mathcal{S},\mathcal{T})$ is a fingerprint of a pair of cross $t$-intersecting families, and there does not exist  $Z\in\binom{[n]}{t+2}$ such that $\mathcal{S},\mathcal{T}\subseteq\binom{Z}{t+1}$. Then $|\mathcal{S}|,|\mathcal{T}|\geq2$, and there exists $I\in\binom{[n]}{t-1}$ and two families $\mathcal{P},\mathcal{Q}\subseteq\binom{[n]\setminus I}{2}$ such that 
	$$\mathcal{S}\subseteq\{I\cup P:P\in\mathcal{P}\}\;\;\mbox{and}\;\;\mathcal{T}\subseteq\{I\cup Q:Q\in\mathcal{Q}\},$$
	where $(\mathcal{P},\mathcal{Q})$ is isomorphic to one of the following pairs:	
	\begin{itemize} 
		\item[\rm(i)]$(\{\{1,2\},\{3,4\},\{1,4\}\},\;\{\{1,3\},\{2,4\},\{1,4\}\})$,
		\item[\rm(ii)]$(\{\{1,2\},\{3,4\},\{1,4\},\{2,3\}\},\;\{\{1,3\},\{2,4\}\})$,
		\item[\rm(iii)]$(\{\{1,3\},\{2,4\}\},\;\{\{1,2\},\{3,4\},\{1,4\},\{2,3\}\})$.
	\end{itemize}
\end{lemma}
By Lemma \ref{lemmastrue-alike}, there is a partial permutation
$I$ of size $t-1$ and four elements $x,y,z,w$ outside $I$ such that the terminal pair $(\mathcal{X}_N,\mathcal{Y}_N)$ is contained in one of 
\begin{align*}
	&(\{Ixy,Izw,Ixw\},\{Ixz,Iyw,Ixw\}), (\{Ixy,Izw,Ixw,Iyz\},\{Ixz,Iyw\}),
\end{align*}
or the reverse of the second pair, where
$Iab:=I\cup\{a,b\}$.

In each case, after interchanging the two sides if necessary,
there is a partial permutation $X$ of size $t$ contained in all but at most
one member of $\mathcal S_N$. It follows that
\begin{align*}
\gamma_t(\mathcal F)&\leq|\mathcal F\setminus\mathcal F[X]|\leq\left|\bigcup_{i=0}^{N-1}\mathcal F[\mathcal X_i]\right|+(n-t-1)!+R\\
&\leq(1+0.1t)(n-t-1)!+R<\gamma_t(\mathcal F_1(n,t)).
\end{align*}
This finishes the proof. {\hfill $\square$}\vspace{1em}

\noindent{\bf Proof of Theorem \ref{thmdiv}.}\;For $t\geq2$, the theorem follows immediately by applying Theorem \ref{thmcrossdiv} with $\mathcal G=\mathcal F$. So we suppose $t=1$. We may assume that $\mathcal F$ is
	non-trivial and adopt the notation in the proof of Theorem \ref{thmcrossdiv}. 
	
	Set $\mathcal G=\mathcal F$, apply Theorem
	\ref{thmspreadapp} only once, and put $\mathcal T=\mathcal S$. 
	Moreover, Algorithm \ref{algo} may be performed diagonally, so that $\mathcal S_i=\mathcal T_i$ and $\mathcal X_i=\mathcal Y_i$ for every $i$. All estimates in the preceding proof remain valid for $t=1$. The
	only case requiring attention is $N=q-2$. In this case $\mathcal C:=\mathcal S_N=\mathcal X_N$ is an
	intersecting family of $2$-sets, each of which is a
	minimal $1$-cover of $\mathcal C$. Then its members have no common element, and so
	$\mathcal C=\binom Z2$ for some $3$-set $Z$.
	Since all three pairs in $Z$ are partial permutations,
	so is $Z$.
	
	If $\mathcal F\subseteq\mathcal F_1(Z)$, then the monotonicity of $\gamma_1$ gives
	$\gamma_1(\mathcal F)\leq\gamma_1(\mathcal F_1(Z))
	=(n-2)!-(n-3)!$.
	Otherwise, choose
	$L\in\mathcal F\setminus\mathcal F_1(Z)$ and $z\in Z$
	such that $L\cap Z\subseteq\{z\}$.
	Every member of $\mathcal F[Z\setminus\{z\}]$ must meet
	$L$ outside $Z\setminus\{z\}$.
	The same counting argument as in the proof gives
	$|\mathcal F[Z\setminus\{z\}]|
	\leq(n-2)!-d_{n-2}\leq0.67(n-2)!$. Then $\mathcal F[\mathcal C]\setminus\mathcal F[\{z\}]
	\subseteq\mathcal F[Z\setminus\{z\}]$,
	Lemma \ref{lemmafin-stru0} and (\ref{equremainder}) yield
	\[
	\gamma_1(\mathcal F)
	\leq|\mathcal F\setminus\mathcal F[\{z\}]|
	\leq0.77(n-2)!+R
	<(n-2)!-(n-3)!,
	\]
	where the last inequality follows from 
	$R=n^{-2}(n-1)!$ and $n\geq400$. This also proves the assertion about optimal families. {\hfill$\square$}
\section{The alternating group}\label{secalternating}
The proof of Theorem \ref{thmcrossAn} follows the same lines as that of Theorem \ref{thmpurelycombin} (ii). We therefore give an abbreviated proof, indicating the necessary modifications and omitting repeated details.

Let $e_s$ and $o_s$ denote the numbers of even and
odd derangements of $[s]$, respectively. Put
\begin{equation}\label{equhAn}
	h_A(n,t):=
	\frac{(n-t)!}{2}-o_{n-t}-o_{n-t-1}.
\end{equation}

%We first describe the extremal families. 
To prove Theorems \ref{thmcrossAn} and \ref{thmAn}, we again identify $A_n$ with a subfamily of $\binom{[n]^2}{n}$. Let us record several facts in the following lemma. For the sake of simplicity, we shall apply Theorem \ref{thmcrossspreadapp}
directly to families in $A_n$, that is, regard them as subfamilies of $S_n$. Nevertheless, we note that $A_n$ itself is also highly spread (see Lemma \ref{lemmacountingalt} (i)). So the method may also be applied directly to $A_n$, with the remainder improved by a factor of two. 
\begin{lemma}\label{lemmacountingalt}
The following hold.
\begin{itemize}
\item[\rm(i)]$\Delta_{n}(A_n)=\Delta_{n-1}(A_n)=\Delta_{n-2}(A_n)=1$, and $\Delta_i(A_n)=\frac12(n-i)!$ for $i\leq n-3$. In particular, $A_n$ is weakly $(r,q)$-spread for
$r\leq(n-q)/e$ and $q\leq n-3$.
\item[\rm(ii)]$\min\{e_{m},o_{m}\}\geq0.15m!$ for $m\geq20$.
\item[\rm(iii)]$h_A(n,t)=0.5((n-t)!-d_{n-t}-d_{n-t-1}+(-1)^{n-t-1})$. In particular, $0.3(n-t)!\leq h_A(n,t)\leq0.33(n-t)!$ for $n-t\geq20$.
\end{itemize}
\end{lemma}
\begin{proof}
The proof of (i) is trivial. Let us prove (ii) and (iii). First, by the definition of $e_m$ and $o_m$, of course $e_m+o_m=d_m$. Next, with $\sigma$ ranging over all derangements in $S_m$,
$$e_m-o_m=
	\sum_{\sigma}
	\operatorname{sgn}(\sigma)={\rm det}(J_m-I_m)=(-1)^{m-1}(m-1).$$
%where \(J_m\) is the all-ones matrix and \(I_m\) is the identity matrix.
 Hence $e_m=(d_m+(-1)^{m-1}(m-1))/2$	and $o_m=(d_m-(-1)^{m-1}(m-1))/2$. The remaining assertions follow from these identities and the standard estimates for derangements.
\end{proof}
\begin{lemma}\label{lemmaindalt}
	Let $t\geq1$, $n\geq400t$ and $q=t+\left\lceil3\ln\binom{n}{t}\right\rceil$. Let $\mathcal{D}\subseteq A_n$, and let $X$ be a partial permutation of $[n]$ of size $t$. Suppose that $L$ is a $t$-cover of $\mathcal{D}$ and $|X\cap L|<t$.
	\begin{itemize}
		\item[\rm(i)]If $|L|\leq q$, then $|\mathcal{D}[X]|\leq0.5q(n-t-1)!$.
		\item[\rm(ii)]If $L\in A_n$, then $|\mathcal{D}[X]|\leq h_A(n,t)$, with equality precisely if $|X\cap L|=t-1$ and  $\mathcal{D}[X]$ consists of all $F\in A_n$ with $X\subseteq F$ and $F\cap(L\setminus X)\neq\emptyset$. Moreover, $|\mathcal{D}[X]|\leq0.25(n-t)!$ for $|X\cap L|\leq t-2$.
		
		\item[\rm(iii)]Suppose that $L\in A_n$ and $|X\cap L|=t-1$, and that $|X\cap K|=t-1$ for another $t$-cover $K\in A_n$ of $\mathcal{D}$ with $L(j)\neq K(j)$ for some $j\in[n]\setminus(X^{(1)}\cup\{k\})$, where $X^{(1)}=\{x:(x,y)\in X\}$ and $k$ is the index in $[n]\setminus X^{(1)}$ with  $X\setminus L=\{(X^{-1}(L(k)),L(k))\}$. Then $|\mathcal{D}[X]|\leq h_A(n,t)-\min\{e_{n-t-1},o_{n-t-1}\}$.
	\end{itemize}
\end{lemma}
The proof is essentially the same as that of Lemma \ref{lemmaind}. 
%We note that ther term $\min\{e_{n-t-1},o_{n-t-1}\}$ comes from estimating the size of $$\{\sigma\in A_n:\sigma(i)=i\;\mbox{for all}\;i\in[t]\cup\{j\},\;\sigma(i)\neq K(i)\;\mbox{for all}\;i\in[t+1,n]\setminus\{j\}\}.$$
\begin{lemma}\label{lemmatcoverkey-permalt}
	Let $\mathcal{D}\subseteq A_n$, and let $\mathcal{C}$ be a
	collection of $t$-covers of $\mathcal{D}$, each of size at most
	$\ell$. If $\tau_t(\mathcal{C})\geq m$ and $n\geq m+\ell+1$, then
	\[
	\Delta_t(\mathcal{D})
	\leq
	\frac12(\ell-t+1)^{m-t}(n-m)!.
	\]
\end{lemma}
The proof is essentially the same as that of Lemma \ref{lemmatcoverkey-perm}.\vspace{1em}

\noindent{\bf Proof of Theorem \ref{thmcrossAn}}.\;Put $q:=t+\lceil3\ln\binom nt\rceil$ and $R:=\binom nt^{-2}(n-t)!$. Apply Theorem \ref{thmcrossspreadapp} to $(\mathcal{F},\mathcal{G})$ and let
	$(\mathcal{S},\mathcal{T})$ be the resulting pair. So $\max\{|\mathcal{F}\setminus\mathcal{F}[\mathcal{S}]|,|\mathcal{G}\setminus\mathcal{G}[\mathcal{T}]|\}\leq R$. Our goal is to prove that, if the pair does not conform to the optimum described, then $|\mathcal{F}||\mathcal{G}|<(h_A(n,t)+t)^2$. To prove this, it is convenient to use the bound $h_A(n,t)\geq0.3(n-t)!$ given in Lemma \ref{lemmacountingalt} (iii).
	
If $\mathcal{S}=\emptyset$ or $\mathcal{T}=\emptyset$, one of
	$\mathcal{F},\mathcal{G}$ has size at most $R$. Since a member
	of the opposite family is a $t$-cover, the other one has size at most $0.5\binom{n}{t}(n-t)!$. Then $|\mathcal{F}||\mathcal{G}|\leq0.5\binom{n}{t}(n-t)!R<h_A(n,t)^2$. 
	
	In what follows, suppose that $\mathcal{S},\mathcal{T}\neq\emptyset$. Perform
	Algorithm \ref{algo} to $(\mathcal{S},\mathcal{T})$. Let $N$ be the number of rounds and $(\mathcal{S}_{i},\mathcal{T}_{i},\mathcal{X}_{i},\mathcal{Y}_i)$ $(i\leq N)$ be the output families.
	
	If $N<q-t$, then by the same argument as in the proof of Lemma \ref{lemmagroupproduct} using Lemmas \ref{lemmadecomp} and \ref{lemmatcoverkey-permalt}, we get 
	\[
	|\mathcal F[\mathcal S]|\,|\mathcal G[\mathcal T]|
	\le\frac14\cdot\left(\frac2{15}+\frac{e^{1/5}}{20}\right)
	((n-t)!)^2
	<0.05((n-t)!)^2,
	\]
and then obtain 	
	\begin{align*}
		|\mathcal{F}||\mathcal{G}|&\leq|\mathcal{F}[\mathcal{S}]||\mathcal{G}[\mathcal{T}]|+R(|\mathcal{F}[\mathcal{S}]|+|\mathcal{G}[\mathcal{T}]|)+R^2\\
		&<0.05((n-t)!)^2+2R\cdot0.5\binom{n}{t}(n-t)!+R^2<h_A(n,t)^2.
	\end{align*}
	
	We may therefore assume that $N=q-t$. By Lemma \ref{lemmafin-stru2}, there is a partial permutation $X$ of
	size $t$ such that
	$\mathcal{S}_{q-t-1}=\mathcal{T}_{q-t-1}=\{X\}$. Let $j$ be the
	smallest index for which
	$\mathcal{S}_j=\mathcal{T}_j=\{X\}$.
	
	If $j>0$, the argument from Case 1 in the proof of Lemma \ref{lemmafin-stru2} yields (after possibly interchanging the two
	families) that 
\begin{align*}
|\mathcal{F}|&\leq\frac{q+0.1t}{2}(n-t-1)!+R\;\;\;\mbox{and}\;\;|\mathcal{G}|\leq\frac12(n-t)!+0.05t(n-t-1)!+R.
\end{align*}
Since $n\geq16q$, their product is strictly smaller than $h_A(n,t)^2$.
	
It remains to consider the case of $j=0$. Then
	$\mathcal{S}_0=\mathcal{T}_0=\{X\}$, and hence
	\[
	\max\{
	|\mathcal{F}\setminus\mathcal{F}[X]|,
	|\mathcal{G}\setminus\mathcal{G}[X]|
	\}
	\leq R.
	\]
	Both sets on the left are non-empty, since
	$\mathcal{F}$ and $\mathcal{G}$ are non-trivial.
	
	Suppose by symmetry that some
	$L\in\mathcal{F}\setminus\mathcal{F}[X]$ satisfies
	$|L\cap X|\leq t-2$. Choosing any
	$M\in\mathcal{G}\setminus\mathcal{G}[X]$ and applying
	Lemma \ref{lemmaindalt} (ii) in both directions gives $
	|\mathcal{G}[X]|\leq\frac14(n-t)!$ and $|\mathcal{F}[X]|\leq h_A(n,t)$. Then 
	$|\mathcal{F}||\mathcal{G}|\leq(|\mathcal{F}[X]|+R)(|\mathcal{G}[X]|+R)<h_A(n,t)^2$. We may therefore assume that $|F\cap X|\geq t-1$ for all $F\in\mathcal{F}\cup\mathcal{G}$.
	
	Choose $\alpha\in\mathcal{F}\setminus\mathcal{F}[X]$. We may suppose, up to a double translation preserving the parity, that
	$X=\{(1,1),\ldots,(t,t)\}$ and that $\alpha$ fixes
	$2,\ldots,t$ but not $1$. 
	
	Choose $M\in\mathcal{G}\setminus\mathcal{G}[X]$, and let
	$i_0$ be the unique point of $[t]$ not fixed by $M$. Put
	$\beta:=(1\ i_0)M(1\ i_0)$. Then we can prove the following claim using Lemma \ref{lemmaindalt} (iii) (by the same argument as in the proof of  Claim \ref{claimlemmafin-stru21}).
	\begin{claim}\label{claimalt}
		The following hold.
		\begin{itemize}
			\item[\rm(i)]Either $|\mathcal G[X]|\leq h_A(n,t)-\min\{e_{n-t-1},o_{n-t-1}\}$ or 
			$\mathcal F\setminus\mathcal F[X]
			\subseteq\{(1\ i)\alpha(1\ i):i\in[t]\}$.
			\item[\rm(ii)]Either $|\mathcal F[X]|\leq h_A(n,t)-\min\{e_{n-t-1},o_{n-t-1}\}$ or $\mathcal G\setminus\mathcal G[X]
			\subseteq\{(1\ i)\beta(1\ i):i\in[t]\}$.
		\end{itemize}
	\end{claim}
Suppose first that the first alternative in either (i) or (ii) holds. Without loss of generality, suppose $|\mathcal G[X]|\leq h_A(n,t)-\min\{e_{n-t-1},o_{n-t-1}\}$. This together with Lemma \ref{lemmacountingalt} and the fact that $R=\binom{n}{t}^{-2}(n-t)!$ gives
$$|\mathcal{G}|\leq h_A(n,t)-\min\{e_{n-t-1},o_{n-t-1}\}+R<h_A(n,t)-R.$$
The other one has size at most $h_A(n,t)+R$. It follows that 
	\begin{align*}
		|\mathcal{F}||\mathcal{G}|<(h_A(n,t)+R)(h_A(n,t)-R)<h_A(n,t)^2.
	\end{align*}
Finally, suppose $\mathcal F\setminus\mathcal F[X]
\subseteq\{(1\ i)\alpha(1\ i):i\in[t]\}$ and $\mathcal G\setminus\mathcal G[X]
\subseteq\{(1\ i)\beta(1\ i):i\in[t]\}$. Then we derive $\mathcal F\subseteq\mathcal H(\alpha,\beta)\cap A_n$
and
$\mathcal G\subseteq\mathcal H(\beta,\alpha)\cap A_n$, and thus $|\mathcal{F}||\mathcal{G}|\leq(h_A(n,t)+t)^2$. It is easy to characterize equality. {\hfill $\square$}
\section*{Acknowledgments}
B. Lv is supported by the National Natural Science Foundation of China (12571347 \& 12131011), and the Beijing Natural Science Foundation (1252010). 
\bigskip

\noindent{\bf Declaration on the use of AI.}\;The authors used AI tools to assist with improving the language and checking technical details. All AI-assisted
revisions were reviewed and approved by the authors before being
incorporated into the manuscript. The authors take full responsibility
for the accuracy, integrity, and originality of the manuscript.

\addcontentsline{toc}{chapter}{Bibliography}

{
	}
\end{document}